\documentclass[10pt,a4paper,reqno]{amsart}
\usepackage{latexsym}
\usepackage{amsmath,amsfonts,amssymb,epsfig,color}
\usepackage{amsthm}
\usepackage{subfigure}
\usepackage{color}
\definecolor{greencite}{rgb}{0.2,0.6,0.2} 
\definecolor{bluformula}{rgb}{0.1,0.2,0.6}
\usepackage[colorlinks=true,linkcolor=bluformula,urlcolor=blue,citecolor=greencite]{hyperref}
 \allowdisplaybreaks

\def\R{\mathbb R}
\def\N{\mathbb N}

\def\de{\delta}
\def\e{\varepsilon}

\def\vphi{\varphi}

\def\Om{\Omega}

\def\pa{\partial}
\newcommand{\W}{\widetilde W}
\newcommand{\F}{\widetilde F}

\def\R{\mathbb R}

\def\vphi{\varphi}

\def\Om{\Omega}
\def\N{\mathbb N}

\def\e{\varepsilon}

\def\pa{\partial}

\newcommand{\wto}{\rightharpoonup}
\newcommand{\beq}{\begin{equation}}
\newcommand{\eeq}{\end{equation}}

\newcommand{\medintinrigo}{-\kern -,315cm\int}
\newcommand{\medint}{-\kern -,375cm\int}
\newtheorem{theorem}{Theorem}[section]
\newtheorem{definition}[theorem]{Definition}
\newtheorem{lemma}[theorem]{Lemma}
\newtheorem{corollary}[theorem]{Corollary}
\newtheorem{proposition}[theorem]{Proposition}
\newtheorem*{theorem*}{Theorem}
\newtheorem{remark}[theorem]{Remark}
\numberwithin{equation}{section}
\begin{document}

\title[Geometrically induced phase transitions]
{Geometrically induced phase transitions in two-dimensional dumbbell-shaped domains}
\author{ M. Morini \& V. Slastikov}
\address[M.\ Morini]{Universit\`a di Parma, Parma, Italy}
\email[Massimiliano Morini]{massimiliano.morini@unipr.it}
\address[V.\ Slastikov]{University of Bristol, Bristol, UK}
\email[Valeriy Slastikov]{Valeriy.Slastikov@bristol.ac.uk}
\begin{abstract} 

We continue the analysis, started in \cite{MorSla}, of a two-dimensional  non-convex variational problem, motivated by  studies on magnetic domain walls trapped by thin necks. The main focus is on the impact of extreme geometry on the structure of local minimizers representing the transition between two different constant phases. We address here the case of general non-symmetric dumbbell-shaped domains with a small constriction and general multi-well potentials. Our main results concern the existence and uniqueness  of non-constant local minimizers, their full classification in the case of convex bulks, and the complete description of their asymptotic  behavior, as the size of the constriction tends to zero. 
\end{abstract}

\maketitle

\section{Introduction}
In this paper we continue the study started in  \cite{KV, MorSla} of the local minimizers of the following non-convex energy functional 
\beq\label{ener}
F(u,\Omega_\e) = \frac{1}{2} \int_{\Omega_\e} |\nabla u|^2 \, dx + \int_{\Omega_\e} W(u) \, dx, 
\eeq
where $\Omega_\e \subset \R^n$ is a dumbbell shaped domain with a small neck (see Figure~\ref{fig:oe}), $W(\cdot)$ is a multi-well potential, and $\e \ll 1$ is a small parameter related to the size of the neck. 

\begin{figure}[htbp]
 \begin{center}
 \includegraphics[scale=0.5]{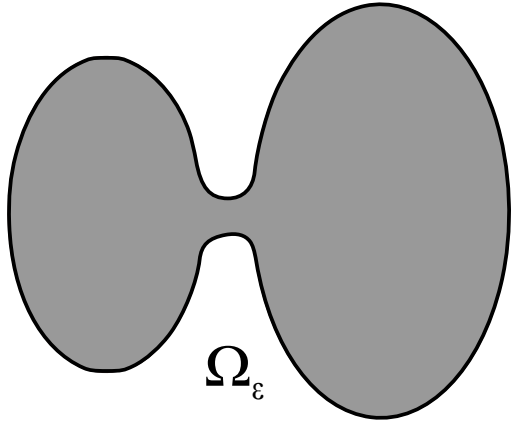}
 \caption{A dumbbell-shaped domain $\Omega_\e$.}
 \label{fig:oe}
 \end{center}
 \end{figure} 

We recall that a physical motivation comes from the investigation of the so-called {\it geometrically constrained walls} and the magnetoresistance properties of thin films with a small constriction. Indeed, if the thin film has cross section along the $xy$-plane given by a domain as in Figure~\ref{fig:oe}, and the magnetization $m$ is allowed to vary only in the $yz$-plane (see Figure~\ref{fig:3D});
 i.e., 
$$
m=(0, \cos u, \sin u ),
$$
with  preferred directions $m=(0, \pm 1, 0)$\footnote{Here $u$ represents the angle between $m$ and the $y$-axis. } (this assumption correspond to the case of uniaxial ferromagnet),   then the magnetostatic interaction can be ignored and the stable magnetic structures are  described by the local minimizers of a non-convex energy of the form \eqref{ener},  with $W(u)\approx \sin^2u$. 
\begin{figure}[htbp]
 \begin{center}
 \includegraphics[scale=0.3]{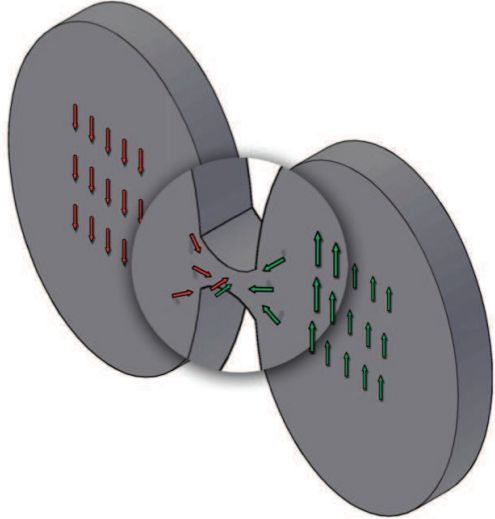}
 \caption{A thin micromagnetic film, with the arrows representing the magnetization. The magnified region is the magnetic domain wall.}
 \label{fig:3D}
 \end{center}
 \end{figure} 
 One then wants to study the nonconstant local minimizers of  \eqref{ener} representing the transition from the constant state $(0, -1, 0)$ in one bulk to the constant state $(0, 1, 0)$ in the other bulk.
 
Following  the pioneering work by Bruno \cite{Bruno}, the study of geometrically constrained walls has attracted   the interest of the physical community from  both the theoretical \cite{MOP, Chen} and the experimental points of view \cite{Harsh, Jub, Sasaki, Tatara}. Bruno noticed that when the size of the constriction becomes very small the neck will be the preferred location for a {\it domain wall}, that is, the transition layer between two regions of (almost) constant magnetization. He also observed that under these circumstances  the impact of the geometry of the neck on the structure of the wall profile becomes dominant and produces a limiting behavior that is independent of the material parameters (whence the name of {\it geometrically constrained} walls or {\em geometrically induced phase transitions}). 

%

When the size $\e$ of the constriction is very small, we may regard the neck as a {\em singular perturbation} of the domain given by the disjoint union of the two bulks.  There exists an extensive mathematical literature devoted to a study of the properties of solutions to nonlinear partial differential equations in singularly perturbed domains, see for instance \cite{Arrieta,  Murat, Dan1,  Daners, HV, Jimbo1, Jimbo2, Jimbo3, KV, MorSla,  RSS}. 
Apart from the directly relevant papers \cite{KV, MorSla} the closest in spirit to this work is that  of Jimbo.  In the series of papers \cite{Jimbo1, Jimbo2,  Jimbo3} he uses PDE methods to study the asymptotic behaviour of the solutions of semilinear elliptic problems for $n$-dimensional dumbbell shaped domain ($n\geq 2$) with a rotationally symmetric neck of fixed length and  shrinking  in the radial direction.

As already mentioned, our work is closely related to  \cite{KV, MorSla}. In \cite{KV} a rigorous study of {\it geometrically constrained walls}  was undertaken in the three-dimensional case. 
The authors constructed a suitable family $u_\e$ of non-trivial local minimizers of $F(\cdot, \Omega_\e)$ with the choice  $W(u):=(u^2-1)^2$ and investigated their asymptotic behavior using variational methods  and $\Gamma$-convergence arguments. This behavior was shown to strongly depend on the size of the neck, specifically on the ratio between the radius $\delta$ of the neck  and its  length $\e$. Three asymptotic regimes were identified, leading to three different limiting problems:
\begin{itemize}
\item[(a)] the {\it thin neck} regime, corresponding to $\displaystyle \frac{\delta}{\e} \to 0$;
\item[(b)] the {\it  normal neck} regime, corresponding to $\displaystyle \frac{\delta}{\e} \to l =cost$;
\item[(c)] the {\it thick neck} regime, corresponding to $\displaystyle \frac{\delta}{\e} \to \infty$.
\end{itemize}
The findings of \cite{KV}  show that in the thin neck regime the wall profile is asymptotically confined  inside the neck and its limiting one-dimensional  behavior depends only on the geometry of the neck. This is the only regime where the one-dimensional ansatz considered in \cite{Bruno} turns out to be  correct. Instead, in the normal neck regime the asymptotic profile is three-dimensional and spreads into the bulks. Finally, in the  thick neck regime  the asymptotic problem is independent of the neck geometry and the full transition between the two states of constant magnetization occurs outside of the neck.


The variational methods introduced in \cite{KV} do not apply to the two-dimensional case, where the logarithmic slow decay of the fundamental solution significantly affects the qualitative behavior of local minimizers. This problem was treated in \cite{MorSla} in the case when $\Omega_\e \subset \R^2$ is a dumbbell shaped domain symmetric with respect to the $y$-axis and $W$ is an even double-well potential  with the  two symmetric wells located at $-1$ and $1$ (and satisfying some additional  structure assumptions of technical nature). More precisely, in \cite{MorSla} we have constructed a particular family $(u_\e)$ of  local minimizers, odd with respect to the $x$-variable, and asymptotically converging to $1$ on the right bulk and to $-1$ on the left bulk, and we studied their asymptotic behavior as $\e\to 0$. 
 The result of this investigation  shows that  the two-dimensional case displayis a richer variety of asymptotic regimes. In particular, in addition to the normal and thick neck regimes, we found  out that the thin neck regime subdivides into three further subregimes:
\begin{itemize}
\item {\it the subcritical thin neck} regime, corresponding to $\displaystyle \frac{\delta |\ln \delta|}{\e} \to 0$;
\item {\it the critical thin neck} regime, corresponding to $\displaystyle \frac{\delta |\ln \delta|}{\e} \to l=const$;
\item {\it the supercritical thin neck} regime, corresponding to $\displaystyle \frac{\delta |\ln \delta|}{\e} \to \infty$.
\end{itemize}
In all cases, the limiting behavior  turns out to be nonvariational and  can be described in terms of  elliptic problems on suitable unbounded domains, with prescribed behavior at infinity. This is the reason why the  approach introduced in \cite{MorSla} is based on PDE methods rather then $\Gamma$-convergence techniques. There, the main idea is  to exploit the Maximum Principle in order to construct precise lower and upper barriers for the given local minimizers, which allow us to capture their asymptotic behavior.  Nevertheless, these constructions heavily rely on the symmetry of $\Om_\e$ and the fact that $u_\e=0$ on the middle vertical segment  $\{x=0\}\cap \Om_\e$. 

The main  questions left open in \cite{MorSla} are: (a) Is the constructed family $(u_\e)$ the {\em unique} family of nontrivial local minimizer of $F_\e$ asymptotically connecting the constant states $-1$ and $1$? (b) Can the analysis of \cite{MorSla} be extended to the case of non-symmetric domains? We address these issues in the present paper. 

\vskip 0.5cm

We are now in a position to describe our results in more detail, referring to the next sections for the precise statements. We assume $\Omega_\e$ to be a dumbbell shaped domain consisting of two bulks $\Omega_\e^l = \Omega_l - (\e, 0)$ and $\Omega_\e^r= \Omega_r + (\e, 0)$ not necessarily symmetric and connected by a small neck $N_\e$. The dimensions of the neck are governed by two small parameters $\e$ and $\de$, corresponding to its length and height, respectively. We consider general  multi-well potentials $W$ of class $C^2$, with isolated wells. Our main findings can be summarized as follows: 
\begin{enumerate}
\item (existence): we prove that for any pair $\alpha\neq \beta$ of wells of $W$,  there exists  a family $(u_\e)$ of non-constant local minimizers of $F(\cdot, \Om_\e)$, which asymptotically connect the constant states $\alpha$ and $\beta$; i.e., $u_\e\approx \alpha$ on one bulk and $u_\e\approx \beta$ on the other, for $\e$ small enough; 
\item (uniqueness): we show that for given $\alpha$, $\beta$, the corresponding family of non-constant local minimizers as in (1) is  {\em unique};
\item (classification): we show that the family of non-constant local minimizers considered in the previous items  {\em exhaust all} the possible local minimizers of $F(\cdot, \Om_\e)$ for $\e$ small enough, provided that the bulks $\Omega^l$ and $\Omega^r$ are {\em convex} and regular enough;
\item (asymptotics): we identify  the limiting behavior of the families of local minimizers considered in (1) and (2) in all the regimes determined by the scaling  parameters $\e$ and  $\de$.
\end{enumerate}
We will refer to the families of  local minimizers described in (1) as {\em families of nearly locally constant local minimizers}.
The proof of the existence is purely variational and adapts to the present setting an argument developed in \cite{KV}.
In fact, the same argument could be used to establish the following general {\em bridge principle}: if $u^l$ and $u^r$ are isolated local minimizers of $F(\cdot, \Omega^l)$ and $F(\cdot, \Omega^r)$, respectively, then there exists a (unique) family $(u_\e)$ of local minimizers of  $F(\cdot, \Om_\e)$ such that $u_\e\approx u^l$ in the left bulk $\Omega_\e^l$  and  $u_\e\approx u^r$ in $\Omega_\e^r$ for $\e$ small enough (see Remark~\ref{rm:bridge}). We remark that  the non-convexity of $\Om_\e$  is a necessary condition for the existence of non-constant local minimizers (see \cite{CaHo}).
 
  The uniqueness follows from showing that the local minimizers $u_\e$ are in fact isolated  $L^1$-local minimizer of $F(\cdot, \Omega_\e)$. This observation is based on a second variation argument and requires one to carefully track the behavior of the first eigenvalue $\lambda_\e$ of $\partial^2F(\cdot, \Om_\e)$, as $\e\to 0$. 

As shown in \cite{CaHo}, if $\Om$ is regular and convex, then all the stable critical points of $F(\cdot, \Omega)$ are constant.  This fact, properly combined with the existence and uniqueness results described before,  allows us to provide a complete classification of the  stable critical points of $F(\cdot, \Omega_\e)$ for $\e$ small enough, when the bulks  $\Omega^l$ and $\Omega^r$ are convex and regular. See Theorem~\ref{th:classification} for precise statement.

Finally, a few words are in order regarding the study of the asymptotic behavior of  the local minimizers. As mentioned before, the methods and the constructions of \cite{MorSla} were heavily relying on the symmetry assumptions of both $\Om_\e$ and the potential $W$. The lack of symmetry here is overcome  by a careful estimate of the amount of energy $F(u_\e, B_\de)$, which concentrates on  small balls  $B_\de$ of size $\de$ centered at points of the neck. This {\em localization estimate}, which is obtained by  a blow-up argument,  allows us to adapt (and in fact to  simplify) some of the constructions of \cite{MorSla} and to extend all the results to general non-symmetric domains.

The paper is organized as follows. In Section~2 we formulate the problem and describe the assumptions on the domains $\Omega_\e$ and the potential $W$. In Section~3 we prove the existence and uniqueness of families  of nearly locally constant local minimizers, and provide a complete classification of stable critical points in the case of regular convex bulks $\Omega^l$ and $\Omega^r$. Finally,  Section~4 is devoted to the  study of the asymptotic behavior of families  of nearly locally constant local minimizers in the various regimes. We will work out the details only in the normal neck and critical thin neck regimes and only state the results in the remaining regimes, leaving the similar (and in fact easier proofs) to the interested reader.

\section{Formulation of the problem}\label{sec:formulation}

In this section we give the precise formulation of the problem. 
We start by describing the limiting domain.  This will be the disjoint union 
$$
\Om_0=\Om^l\cup\Om^r\,,
$$
where $\Om^l$ and $\Om^r$  are bounded connected open sets of class $C^{1,\gamma}$ for some $\gamma\in (0,1)$, satisfying (see Figure~\ref{fig:limiting}):
\begin{itemize}
\item[(O1):] the origin $(0,0)$ belongs to both $\partial \Om^r$ and $\partial \Om^l$;
\item[(O2):] $\Omega^r$ lies  in the right half-plane $\{x>0\}$, while $\Om^l$ lies in left half-plane $\{x<0\}$.
\end{itemize}
Finally, throughout the paper we will also make the following technical assumption: 
\begin{itemize}
\item[(O3):] there exists $r_0>0$ such that $\partial \Om^r\cap B_{2r_0}(0,0)$  and $\partial \Om^l\cap B_{2r_0}(0,0)$ are flat and vertical.
\end{itemize}
Hypothesis (O3) is not really necessary for the analysis carried out in this paper. We decided to add it in order to avoid some technicalities that would distract from the main new ideas introduced here. All the results
 we are going to prove remain valid also without the additional assumption. Indeed,  if  (O3) does not hold, one can reduce to it by  straightening the boundary through a suitable conformal change of variables  and then construct the same barriers and test functions presented  here, but with respect to the new variables (see, for instance,  \cite[Section 3]{MorSla}).  
\begin{figure}[htbp]
 \begin{center}
 \includegraphics[scale=0.5]{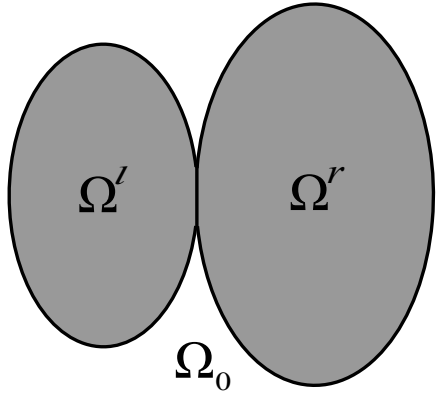}
 \caption{The limiting set $\Om_0$.}
 \label{fig:limiting}
 \end{center}
 \end{figure} 

The profile of the neck after  rescaling is described by
 two functions $f_1, f_2 \, : \, [-1,1] \mapsto (0,+\infty)$ of class $C^{1,\gamma}$ and by
the two small parameters $\e>0$ and $\delta=\de(\e)>0$, which represent the scaling of length and height of the neck, respectively. \emph{As in \cite{MorSla}, $\de$ will always be considered  as depending 
on $\e$, even though, for notational convenience, we will often omit to explicitly write such a dependence.} We also assume throughout
the paper that $\de(\e)\to 0$ as $\e\to 0$.
To describe the $\e$-domain, we set
\beq \label{omega-epsilon}
\Omega_\e = \Omega_\e^l \cup N_\e \cup \Omega_\e^r,
\eeq
where
\beq \label{omega-epsilon-lr}
\Om_\e^r := \Omega^r + (\e, 0)\,,\qquad \Om_\e^l := \Omega^l - (\e, 0)
\eeq
and 
\beq\label{neck}
N_\e := \Bigl\{ (x,y) \, : \, |x|\leq\e,\, -\delta f_2\Bigl(\frac{x}{\e}\Bigr)<y< \delta f_1\Bigl(\frac{x}{\e}\Bigr) \Bigr\}
\eeq
(see Figure \ref{fig:omegae}).
\begin{figure}[htbp]
 \begin{center}
 \centerline{
    $\vcenter{\mbox{\epsfig{file=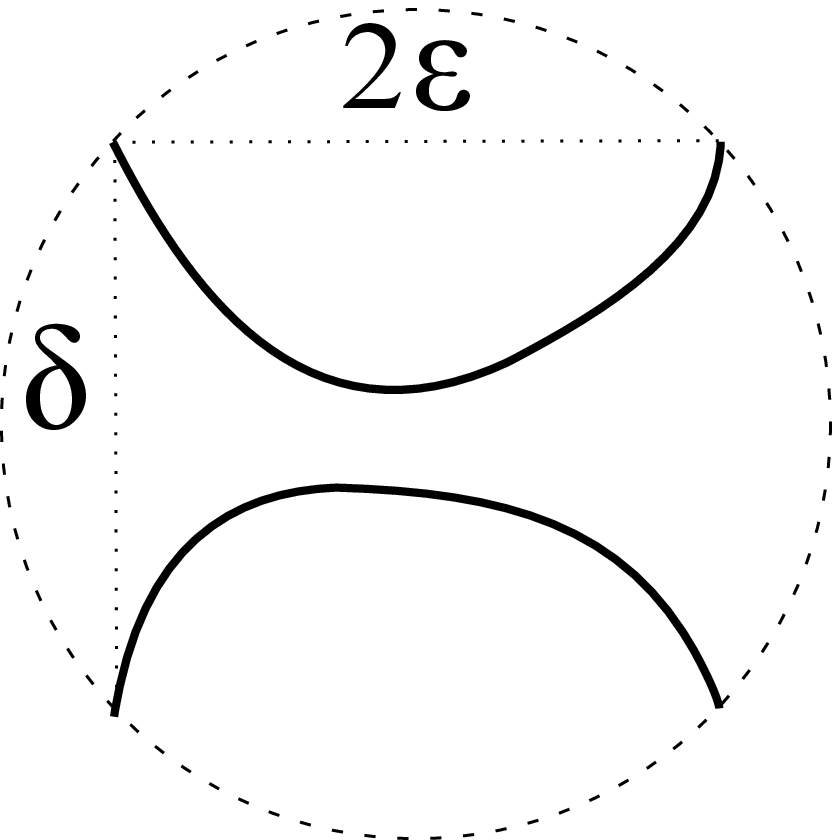, width=0.23\textwidth}}}$
    \hspace{-9cm}
    $\vcenter{\mbox{\epsfig{file=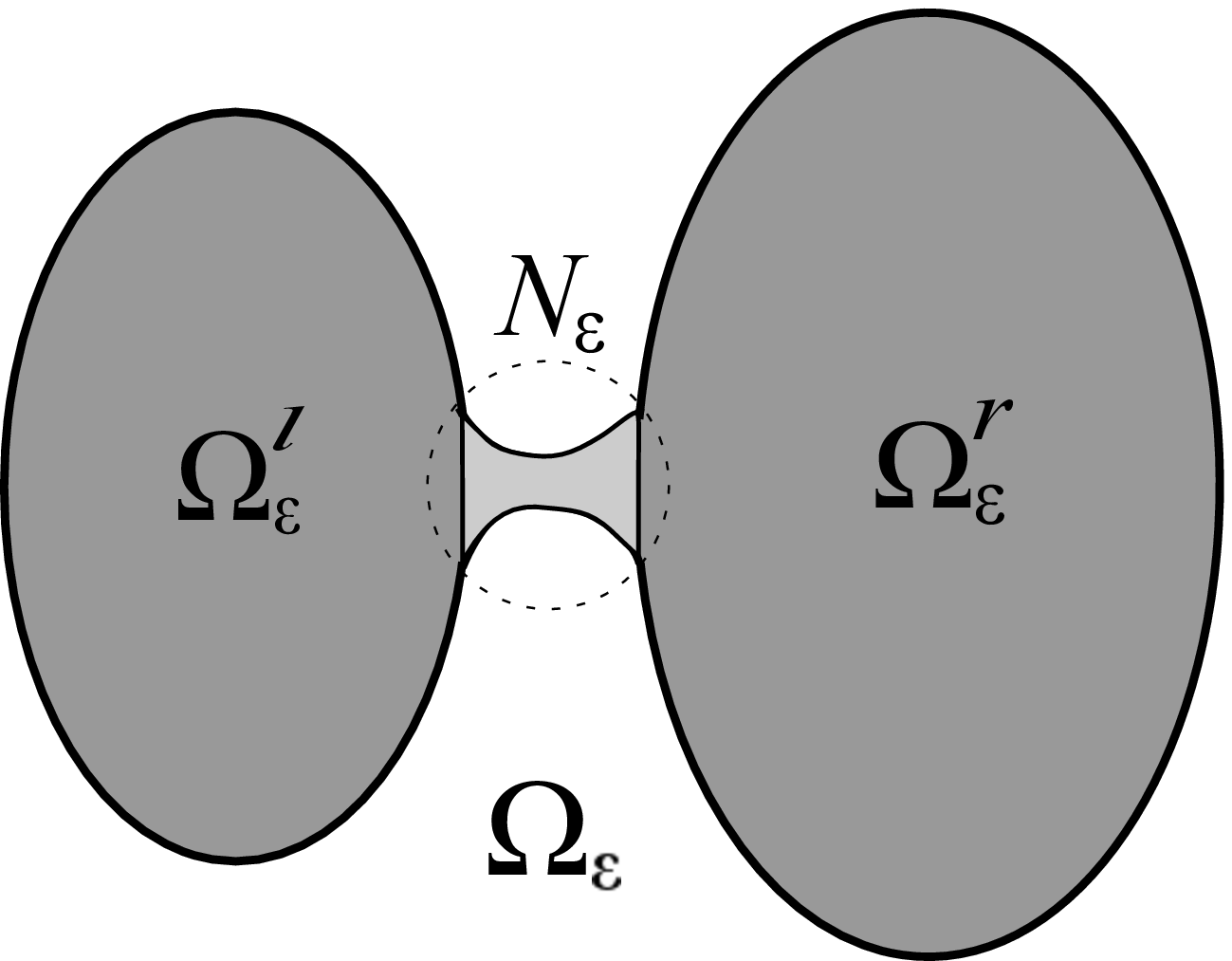, width=0.5\textwidth}}}$
  }
 \caption{The  dumbbell-shaped set $\Om_\e$.}
 \label{fig:omegae}
 \end{center}
 \end{figure} 
 Note that 
 $$
 N_\e=\{(\e x,\de y): (x, y)\in N\}\,,
 $$
 where $N$ is the unscaled neck given by 
 \beq\label{unscaledN}
 N=\{(x,y):x\in[-1,1]\,, -f_2(x)<y<f_1(x)\}\,.
 \eeq
  Finally,  observe that $\Omega_\e$ is a Lipschitz domain. 


The main focus of the paper is the study of a  suitable class of {\em nearly constant  critical points}  of the energy functional 
$$
F(u, \Om_\e) := \frac12\int_{\Omega_\e} |\nabla u|^2\, dxdy + \int_{\Omega_\e} W(u) \, dxdy\,,
$$
defined for all $u\in H^1(\Om_\e)$.  Here, and throughout the paper,  $W:\R\to \R$ is a {\em multi-well potential} with the following
properties (see Figure~\ref{double-well-nonsym}):
\begin{itemize}
\item[(W1)] $W$ is of class $C^2$;
\item[(W2)] the set 
\beq\label{setV}
V:=\{t\in \R:\, W'(t)=0\text{ and }ÊW''(t)>0\}
\eeq
 contains at least two points.
\end{itemize}
Clearly, $V$ represents the set of wells of the potential $W$.
A model case is of course given  by $W(u):=(u-\alpha)^2(u-\beta)^2$.
\begin{figure}[htbp]
 \begin{center}
 \includegraphics[scale=0.3]{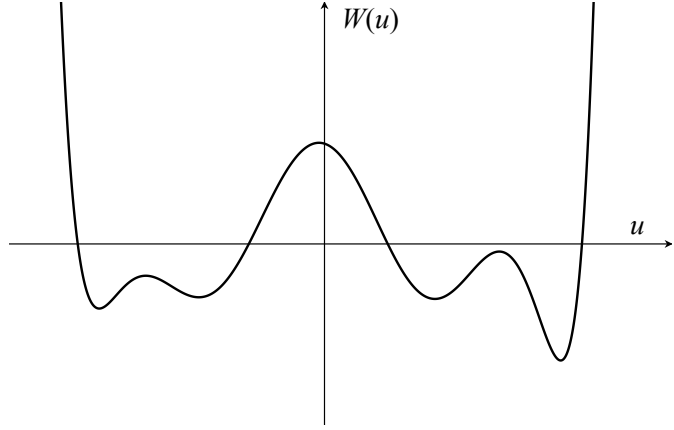}
 \caption{An example of a potential $W(u)$.}
 \label{double-well-nonsym}
 \end{center}
 \end{figure}
 
We recall that a function $u\in H^1(\Om_\e)$ is a critical point for $F(\cdot, \Om_\e)$ if it satisfies 
\beq\label{EL}
\begin{cases}
\Delta u=W'(u) & \text{in $\Om_\e$,}\\
\displaystyle \frac{\pa u}{\pa \nu}=0 & \text{on $\pa\Om_\e$,}
\end{cases}
\eeq
or, equivalently,
\beq\label{ELw}
\int_{\Om_\e}\nabla u\nabla \vphi\, dxdy+\int_{\Om_\e}W'(u)\vphi\, dxdy=0\qquad\text{for all }\vphi\in H^1(\Om_\e)\,.
\eeq
Finally, it is convenient to  extend the definition of $F$ to any subset $\Om\subset\R^2$, by setting
\beq\label{Fanyomega}
F(u, \Om) := \frac12\int_{\Omega} |\nabla u|^2\, dxdy + \int_{\Omega} W(u) \, dxdy\,,
\eeq
for all $u\in H^1(\Om)$. 

\section{Nearly locally constant critical points}
In this paper we are concerned with the existence and the asymptotic behavior of  sequences of critical points 
that are nearly constant, according to the following definition.
\begin{definition}\label{def:gcw}
For $\e>0$ let $u_\e\in H^1(\Om_\e)$ be a critical point of $F(\cdot, \Om_\e)$. We say that the  family $(u_\e)$ is an {\em admissible family of nearly locally constant critical points} if
\begin{itemize}
\item[(a)] there exists $\bar\e>0$ such that $\sup_{0<\e\leq\bar\e}\|u_\e\|_{\infty}=: \overline M<+\infty$;
\item[(b)] there exist constants $\alpha\neq \beta$ belonging to the set $V$ in \eqref{setV}  such that  
\beq\label{alphabeta0}
 \|u_\e-\alpha\|_{L^1(\Om^l_\e)}\to 0 \qquad\text{and}\qquad \|u_\e-\beta\|_{L^1(\Om^r_\e)}\to 0\,,
 \eeq
 as $\e\to 0^+$.
\end{itemize}
\end{definition}
\begin{remark}\label{rm:w3}
Under additional assumptions on the potential $W$, condition (a) in the above definition is automatically satisfied. For instance, this is the case when $W$ satisfies:
\begin{itemize}
\item[(W3)] there exists $\overline M>0$ such that $W'(t)>0$ if $t\geq \overline M$ and $W'(t)<0$ if $t\leq -\overline M$. 
\end{itemize}
Indeed, by the maximum principle one can show that any solution $u$ to \eqref{EL} satisfies $|u|\leq \overline M$.
\end{remark}
We will prove that nearly constant critical points are local minimizers of the energy functional for  $\e$ small enough. Hence, they represent physically observable stable configurations.

 \begin{theorem}\label{th:locmin}
 Let $(u_\e)$ be a family of critical points as in Definition~\ref{def:gcw}. Then, there exist $\e_0>0$ and $\eta_0>0$
 such that 
 for $0<\e\leq\e_0$ 
 \beq\label{eq:locmin}
 F(v, \Om_\e)>F(u_\e, \Om_\e) \qquad\text{for all $v\in H^1(\Om_\e)$ such that $0<\|v-u_\e\|_{L^1(\Om_\e)}\leq \eta_0$.}
 \eeq
Both $\e_0$ and $\eta_0$ depend only on the constants $\alpha$, $\beta$, and  $\overline M$ appearing in Definition~\ref{def:gcw}.
 \end{theorem}	 
 The proof the theorem borrows some ideas from \cite[Lemma 2.2]{MorSla}. Before starting, we recall the following simple Poincar\'e inequality (see  \cite[Proof of Lemma 2.2-Step 2]{MorSla}).
 \begin{lemma}\label{lm:epoinc}
There exists  a constant $C_1>0$ independent of $\e$ such that 
$$
\int_{N^+_\e}|\nabla \vphi|^2\, dxdy\geq \frac{C_1}{\e^2}\int_{N^+_\e}|\vphi|^2\, dxdy 
$$
for all $\vphi\in H^1(N^+_\e)$ satisfying $\vphi=0$ on $\{x=\e\}$, where $N^+_\e:=N_\e\cap\{x>0\}$.
 \end{lemma}
 
\begin{proof}[Proof of Theorem~\ref{th:locmin}]
We split the proof into theree steps.

\noindent {\bf Step 1.} ({\it Positive definiteness of the second variation})
We start by assuming that there exists $M>0$ such that
\beq\label{wcomodo}
|W''(t)|\leq M\qquad\text{for all $t\in \R$.}
\eeq
 Given $u$, $\vphi\in H^1(\Om_\e)$, and $\Om\subset \Om_\e$ we define the second variation of $F(\cdot, \Om)$ at $u$ with respect to the direction
  $\vphi$ as
  $$
  \pa^2 F(u, \Om)[\vphi]:= \frac{d^2}{dt^2}F(u+t\vphi, \Om)|_{t=0}
  =\int_{\Om}|\nabla\vphi|^2\, dxdy+\int_{\Om}W''(u)\vphi^2dxdy\,.
  $$
  Set $\Om_\e^+:=\Om_\e\cap \{x>0\}$. 
 We claim that there exist $\eta^+_0>0$ (independent of $\e$) and $\e_0^+>0$  such that 
\beq\label{pa2>0}
\pa^2 F(v,\Om^+_\e)[\vphi]\geq\frac{W''(\beta)}2\|\vphi\|^2_{L^2(\Om^+_\e)}\qquad\text{for all $\vphi\in H^1(\Om_\e^+)$}
\eeq
provided that $\e\in (0, \e_0^+)$ and $\|v-u_\e\|_{L^1(\Om^+_\e)}\leq \eta_0^+$.
 To this aim, we argue by contradiction by assuming that there exist $\e_n\to 0$ and  $(v_n)$ such that 
 \beq\label{vien}
 \|v_n-u_{\e_n}\|_{L^1(\Om^+_{\e_n})}\to 0
 \eeq
 and for all $n\in\N$
 $$
 \pa^2 F(v_n,\Om^+_{\e_n})[\vphi]<\frac{W''(\beta)}2\|\vphi\|^2_{L^2(\Om^+_{\e_n})}\qquad\text{for some $\vphi\in H^1(\Om_{\e_n}^+)$.}
 $$
 Thus, if we set 
\beq\label{le+}
\lambda^+_n:=\min\left\{\pa^2F(v_n, \Om_\e^+)[\vphi] :\,\vphi\in H^1(\Om_\e^+),\, \|\vphi\|_{L^2(\Om_\e^+)}=1\right\}\,,
\eeq
we have 
\beq\label{limle}
\liminf_{n\to +\infty}\lambda_n^+\leq \frac{W''(\beta)}2\,.
\eeq
We may assume, without generality, that $\liminf_{n\to \infty}\lambda_n^+=\lim_{n\to \infty}\lambda_n^+$. Let $\vphi_n$ be a minimizer for the problem \eqref{le+} corresponding to $\e_n$ and note that 
\beq\label{lab*}
\sup_n\|\vphi_n\|_{H^1(\Om^+_{\e_n})}^2\leq 1+ \sup_n(\lambda^+_{n}+\|W''(u_{\e_n})\|_\infty)<+\infty\,.
\eeq
Thus, in particular, there exists $\vphi\in H^1(\Om^r)$ and a subsequence (not relabeled) such that
\beq\label{wtoloc}
\psi_n:=\vphi_n(\e_n+\cdot, \cdot)\wto \vphi 
\eeq
weakly in $H^1(\Om^r)$. 
We claim that 
\beq\label{svan}
\|\vphi\|_{L^2(\Om^r)}=1\,.
\eeq
To this aim, extend $\vphi_n|_{\Om_{\e_n}^r}$ to a function $\tilde \vphi_n\in H^1(\R^2)$ in such a way 
that 
$$
\|\tilde \vphi_n\|_{H^1(\R^2)}\leq C'\|\vphi_n\|_{H^1(\Om^r_{\e_n})}\,,
$$
with $C'$ independent of $n$, where we recall $\Om^r_{\e_n}=\Om^r+(\e,0)$.
Note that this is possible due to the regularity of $\pa\Om^r$. 

Fix $p>2$. Then,
\beq\label{embedding}
\int_{N^+_{\e_n}}\tilde \vphi_n^2\, dxdy \leq 
\Bigl(\int_{N^+_{\e_n}}\tilde \vphi_n^p\, dxdy\Bigr)^\frac{2}{p}|N_{\e_n}^+|^{1-\frac2p}\leq
c_p\|\vphi_{n}\|^2_{H^1(\Om^+_\e)}|N_{\e_n}^+|^{1-\frac2p}\to 0\,,
\eeq
where we used the imbedding of $H^1(\R^2)$ into $L^p(\R^2)$ and \eqref{lab*}. Moreover, 
\begin{align}\label{step2}
\int_{N^+_{\e_n}}|\nabla \vphi_n|^2\, dxdy &\geq \frac12
\int_{N^+_{\e_n}}|\nabla (\vphi_n-\tilde \vphi_n)|^2\, dxdy-
\int_{N^+_{\e_n}}|\nabla \tilde \vphi_n|^2\, dxdy\\
&\geq \frac{C_1}{\e_n^2}\int_{N^+_{\e_n}}|\vphi_n-\tilde \vphi_n|^2\, dxdy-C_2\,,\nonumber
\end{align}
where in the last inequality we have used Lemma~\ref{lm:epoinc} and again the fact that 
$\sup_n\|\tilde \vphi_n\|^2_{H^1(\R^2)}\leq C_2<+\infty$ thanks to \eqref{lab*}. Since the left-hand side of \eqref{step2}
is bounded, recalling \eqref{embedding}, we deduce
\beq\label{ahecco}
\int_{N^+_{\e_n}} \vphi_n^2\, dxdy\to 0\,.
\eeq
Thus, claim \eqref{svan} follows from  \eqref{wtoloc} observing that  
$\int_{\Om^r} \psi^2_n\, dxdy= 1- \int_{N^+_{\e_n}} \vphi_n^2\, dxdy$. 

Set now $w_n(x,y):= v_n(x+\e_n, y)$ and note that by \eqref{alphabeta0}-(ii) and \eqref{vien} we have 
$w_n\to \beta$ in $L^1(\Om^r)$. Thus, 
 by lower semicontinuity  and recalling also \eqref{le+} and \eqref{wtoloc}, we have
\begin{align*}
\liminf_{n\to\infty}\lambda_{\e_n}^+ & \geq 
\liminf_{n\to\infty}\int_{\Om_{\e_n}^r}|\nabla \vphi_n|^2\, dxdy+\int_{\Om^+_{\e_n}}W''(v_n)\vphi_n^2dxdy\\
&= \liminf_{n\to\infty}\int_{\Om^r}|\nabla \psi_n|^2\, dxdy+\int_{\Om^r}W''(w_n)\psi_n^2dxdy\\
&\geq  \int_{\Om^r}|\nabla \vphi|^2\, dxdy+W''(\beta)\int_{\Om^r}\vphi^2dxdy\geq W''(\beta)\,,
\end{align*}
where the equality is a consequence of \eqref{wcomodo} and \eqref{ahecco}, while the last inequality follows from the definition of $\lambda_0^+$ and \eqref{svan}. The above chain of inequalities contradicts \eqref{limle} and completes the proof of  \eqref{pa2>0}. An entirely similar argument shows that there exist $\eta^-_0>0$ (independent of $\e$) and $\e_0^->0$  such that 
$$
\pa^2 F(v,\Om^-_\e)[\vphi]\geq\frac{W''(\alpha)}2\|\vphi\|^2_{L^2(\Om^-_\e)}\qquad\text{for all $\vphi\in H^1(\Om_\e^-)$}
$$
provided that $\e\in (0, \e_0^-)$ and $\|v-u_\e\|_{L^1(\Om^-_\e)}\leq \eta_0^-$, where $\Om_\e^-:=\Om_\e\cap \{x<0\}$. Thus, setting $\e_0:=\min\{\e_0^-,\e_0^+ \}$, $\eta_0:=\min\{\eta_0^-,\eta_0^+ \}$, and $\lambda_0:=\min\{W''(\alpha),W''(\beta)\}$, we may assert that
\beq\label{pa2>0bis}
\pa^2 F(u_\e,\Om_\e)[\vphi]\geq\frac{\lambda_0}2\|\vphi\|^2_{L^2(\Om_\e)}\qquad\text{for all $\vphi\in H^1(\Om_\e)$}
\eeq
provided that $\e\in (0, \e_0)$ and $\|v-u_\e\|_{L^1(\Om_\e)}\leq \eta_0$.

\noindent {\bf Step 2.} ({\it Conclusion under assumption \eqref{wcomodo}}) Assume \eqref{wcomodo}. Fix $v\in H^1(\Om_\e)$ with $\|v-u_\e\|_{L^1(\Om_\e)}\leq \eta_0$ and
  set 
$f(t):=F(u_\e+t(v-u_\e),\Om_\e)$. Then, for $t\in (0,1)$ by \eqref{pa2>0bis} we have
$$
 f''(t)  = \pa^2 F(u_\e+t(v-u_\e),\Om_\e)[v-u_\e]\geq \frac{\lambda_0}2\|v-u_\e\|^2_{L^2(\Om_\e)}
 $$
 provided that $\e\in (0, \e_0)$.
Hence, also  recalling that $f'(0)=0$ due to the criticality of $u_\e$, we deduce
\begin{align*}
F(v,\Om_\e)&=f(1)=f(0)+\int_0^1(1-t)f''(t)\, dt\\
&\geq F(u_\e, \Om_\e)+\frac{\lambda_0}4\|u_\e-v\|^2_{L^2(\Om_\e)}\,,
\end{align*}
which yields the conlusion of the theorem under assumption \eqref{wcomodo}

\noindent {\bf Step 3.} ({\it The general case})
 We now remove the extra assumption \eqref{wcomodo}. To this aim, let $ \W:\R\to \R$ be a $C^2$ function such that 
  \begin{itemize}
  \item[(a)] $\W=W$ on $[-\overline M, \overline M]$ where $\overline M$ is the constant appearing in condition (a) of Definition~\ref{def:gcw};
  \item[(b)] $\W\leq W$ everywhere;
  \item[(c)] $|\W''|\leq M$ everywhere, for some $M>0$.
  \end{itemize}
  For every $u\in H^1(\Om_\e)$  define
 $$
  \F(u, \Om_\e):=\frac12\int_{\Omega_\e} |\nabla u|^2\, dx + \int_{\Omega_\e} \W(u) \, dx\,,
  $$
  and note that 
  $$
 F(v,\Om_\e)\geq  \F(v,\Om_\e)  \quad \text{for all $v\in H^1(\Om_\e)$ and } \F(u_\e,\Om_\e)=F(u_\e,\Om_\e)\,.
  $$
Then, by the previous step, there exist $\lambda_0>0$, $\eta_0>0$, and $\e_0>0$ such that
$$
 F(v,\Om_\e)\geq  \F(v,\Om_\e)\geq \F(u_\e, \Om_\e)+\frac{\lambda_0}4\|u_\e-v\|^2_{L^2(\Om_\e)}=F(u_\e, \Om_\e)+\frac{\lambda_0}4\|u_\e-v\|^2_{L^2(\Om_\e)}\,,
$$
provided that $\e\in (0, \e_0)$ and $\|v-u_\e\|_{L^1(\Om_\e)}\leq \eta_0$. This concludes the proof of the theorem.
\end{proof}
 \begin{remark}\label{rm:secvar}
 We highlight here the following well-known fact: if $u\in H^1(\Om)\cap L^{\infty}(\Om)$ is a critical point for $F(\cdot, \Om)$ and 
 $$
 \pa^2F(u,\Om)[\vphi]>0\qquad\text{for all $\vphi\in H^1(\Om)\setminus\{0\}$},
 $$
 then $u$ is an isolated local $L^1$-minimizer; i.e, there exists $\eta_0>0$ such that $F(v,\Om)>F(u,\Om)$ for all
 $v\in H^1(\Om)$ with $0<\|v-u\|_{L^1(\Om)}\leq \eta_0$. 
 This fact can be proved with arguments similar to the ones used in the proof of previous  theorem. More precisely, one first observes as before that  \eqref{wcomodo} may be assumed without loss of generality. 
 Then, one shows that the map
 $$
 v\in H^1(\Om)\mapsto \lambda(v):=\min\left\{\pa^2F(v, \Om)[\vphi] :\,\vphi\in H^1(\Om),\, \|\vphi\|_{L^2(\Om)}=1\right\}\
 $$
 is lower semicontinuous with respect to the $L^1$-convergence. This is similar to Step 1 of the previous proof and in fact easier since there is no $\e$-dependence. The conclusion then follows arguing as in Step~2 of the previous proof. 
 \end{remark} 
  In the following we show  that the existence of at least one family of geometrically constrained walls can be proven through a constrained minimization procedure, similar to the one used in \cite[Theorem 3.1]{KV}. For the reader's convenience we provide the full proof. To this aim, 
  for $\alpha$, $\beta\in V$ (see (W2)) and for $\e\in (0,1)$ define
  $$
  u_{0,\e}(x,y):=
  \begin{cases}
  \alpha & \text{if $(x,y)\in \Om_\e^l$,}\\
  \frac{\alpha+\beta}2 & \text{if $(x,y)\in N_\e$,}\\
  \beta & \text{if $(x,y)\in \Om_\e^r$.}\\
  \end{cases}
  $$
  Moreover, for $d>0$ set 
  \beq\label{bde}
  B_{d,\e}:=\{u\in H^1(\Om_\e):\, \|u-u_{0,\e}\|_{L^1(\Om_\e)}\leq d\}\,.
  \eeq
  \begin{theorem}[Existence of nearly locally constant critical points]\label{th:eu}
 For any $\alpha\neq\beta\in V$  there exists an admissible family $(u_\e)$ of nearly locally constant  critical points  as in Definition~\ref{def:gcw}. 
 
\end{theorem}

\begin{proof}
We introduce a potential $\widetilde W$ of class $C^1$, with the following properties:
\begin{enumerate}
\item[(a)] $\widetilde W (u) = W (u)$ for $\min\{\alpha,\beta\} \leq u \leq \max\{\alpha,\beta\}$;
\item[(b)] $\widetilde W' (u) <0$ for $u <\min\{\alpha,\beta\}$ and $\widetilde W' (u) > 0$ for $u > \max\{\alpha,\beta\}$.
\end{enumerate}
Accordingly, we consider the  energy functional 
\beq\label{tildeffe0}
\widetilde F (u, \Omega) := \frac12\int_{\Omega} |\nabla u|^2\, dx + \int_{\Omega} \widetilde W(u) \, dx\,.
\eeq
Let us fix $d>0$ so small that for all $u \in H^1(\Omega^l)$, with $0<\|u -\alpha\|_{L^1(\Omega^l)} \leq d$ we have $\widetilde F (u, \Omega^l) > \widetilde F (\alpha, \Omega^l)$, and for all $v \in H^1(\Omega^r)$, with $0<\|v - \beta \|_{L^1(\Omega^r)} \leq d$ we have   $\widetilde F (v, \Omega^l) > \widetilde F (\beta, \Omega^l)$. This is possible  since the constant functions $\alpha$ and $\beta$ are isolated local minimizers of $\widetilde F (\cdot, \Omega^l) $ and $\widetilde F (\cdot, \Omega^r)$, respectively (see Remark~\ref{rm:secvar}).

Let $u_\e$  be a minimizer of the problem
\beq\label{tildeffe}
\min_{u_\e \in B_{d,\e} } \widetilde F (u, \Omega_\e)\,,
\eeq
where $B_{d,\e}$ is the set defined in \eqref{bde}.
We would like to show that there exists $\e_0>0$ such that for all $\e <\e_0$ the function $u_\e$ is an $L^1$-local minimizer of $\widetilde F (\cdot, \Omega_\e)$. In order to do this we adapt the arguments of \cite[Theorem 1]{KV}. 

Using property (b) of $\widetilde W$ and a truncation argument it is straightforward to show that
\beq\label{limitate}
\min\{\alpha,\beta\} \leq u_\e \leq \max\{\alpha,\beta\}.
\eeq
We also notice that if $u_\e$ lies in the interior of $B_{d,\e}$ then it is an $L^1$-local minimizer of $\widetilde F (u, \Omega_\e)$. In fact, we claim that 
\beq\label{lab**}
\lim_{\e\to 0}\|u_{\e} - u_{0,\e}\|_{L^1(\Omega_\e)} =0\,.
\eeq
Let $M:= \max\{ f_1 (\pm 1), f_2 (\pm 1) \} +1$, $\gamma\in (0,1)$, and consider the following test function
$$
\xi_{\e} (x,y) := \begin{cases}
\alpha & \text{ if $|(x+\e, y)| \geq \de^\gamma$ and $ x < -\e$ } \\
{\alpha + \beta \over 2} - h_\e(x+\e, y) & \text{ if $|(x+\e, y)| < \de^\gamma$ and $ x< -\e$ } \\
{\alpha + \beta \over 2} & \text{ if  $ -\e \leq x \leq \e$ } \\
{\alpha + \beta \over 2} + h_\e(x-\e, y)& \text{ if $|(x-\e, y)| \leq \de^\gamma$ and $ x>\e$ } \\
\beta & \text{ if $|(x-\e, y)| > \de^\gamma$ and $ x>\e$, } 
\end{cases}
$$
where $h_\e \, :\,  \R^2 \to \R$ satisfies
$$
\begin{cases}
\Delta h_\e(x,y) =0 & \text{ for $(x,y) \in  B_{\delta^\gamma}(0,0) \setminus \overline{B_{M \delta}(0,0)}$} \\
h_\e(x,y) = 0 &  \text{ for $(x,y) \in  \bar B_{M \delta}(0,0)$}\\
h_\e(x,y) = \frac{\beta - \alpha}{2} &  \text{ for $(x,y) \in  \R^2 \setminus B_{\delta^\gamma}(0,0)$}.
\end{cases}
$$
Note that the function $h_\e$ in $ B_{\delta^\gamma}(0,0) \setminus\overline{B_{M \delta}(0,0)}$ is explicitly given by
$$
h_\e(x,y)=\frac{\beta-\alpha}{2(\log\de^{\gamma-1}-\log M)}\log\frac{|(x,y)|}{M\de}\,.
$$
It is easy to check that $\| \xi_{\e} -u_{0,\e} \|_{L^1(\Omega_e)} \to 0$ as  $\e \to 0$. Moreover, a direct computation shows 
$$
\lim_{\e \to0}\widetilde F (\xi_{\e}, \Omega_\e)=W(\alpha) |\Omega^l| + W(\beta) |\Omega^r|\,.
$$
Therefore, by the minimality of $u_\e$, we have
\beq \label{lsF}
\limsup_{\e\to0} \widetilde F(u_{\e_k}, \Omega_{\e_k}) \leq \lim_{\e\to0} \widetilde F(\xi_{\e_k}, \Omega_{\e_k}) = W(\alpha) |\Omega^l| + W(\beta) |\Omega^r|\,.
\eeq
Fix now any sequence $\e_k\to 0$ and 
 define
$$
u_k^l (x, y):= u_{\e_k} (x -\e_k, y), \text{ for } (x, y) \in \Omega^l,
$$
$$
u_k^r (x, y):= u_{\e_k} (x +\e_k, y), \text{ for } (x, y) \in \Omega^r.
$$
It is clear that both sequences are bounded in $H^1$ and therefore, up to a subsequence (not relabeled), we may assume $u_k^l \wto u^l_*$  and $u_k^r \wto u^r_*$ weakly in $H_1 (\Omega^l)$ and  $H_1 (\Omega^r)$, respectively, with $\|u^l_*-\alpha\|_{L^{1}(\Om^l)}\leq d$ and $\|u^r_*-\beta\|_{L^{1}(\Om^r)}\leq d$. 
 Recalling \eqref{limitate}, note that
$$
 \widetilde F(u_{\e_k}, \Omega_{\e_k}) \geq \widetilde F(u_k^l, \Omega^l) + \widetilde F(u_k^r, \Omega^r)
 -|N_{\e_k}| \sup_{ |t| \leq \max\{|\alpha|,|\beta|\}}|W(t)|\, \,. 
$$
Thus, using also \eqref{lsF}, we obtain
$$
W(\alpha) |\Omega^l| + W(\beta)|\Omega^r| \geq \liminf \widetilde F(u_{\e_k}, \Omega_{\e_k}) \geq  \widetilde F(u_*^l, \Omega^l) +  \widetilde F(u_*^r, \Omega^r) \geq W(\alpha) |\Omega^l| + W(\beta)|\Omega^r|. 
$$
Since $\alpha$ and $\beta$ are isolated local minimizers of $ \widetilde F(\cdot, \Omega^l)$ and $ \widetilde F(\cdot, \Omega^r)$,  the above chain of inequalities implies  that $u_*^l =\alpha$ and $u_*^r =\beta$. But then, 
$\|u_{\e_k} - u_{0,\e_k}\|_{L^1(\Omega_{\e_k})}\to 0$ and claim \eqref{lab**} is established. 
Thus, $u_\e$ is a local minimizer and, in turn, a critical point  of $\widetilde F(\cdot, \Omega_\e)$ for $\e$ small enough. Recalling  property (a) satisfied by  $\widetilde W$ and \eqref{limitate}, it plainly follows that $u_\e$ is also a critical point of $F(\cdot, \Omega_\e)$. It is now clear that the family $(u_\e)$ satisfies all the properties
 stated in Definition~\ref{def:gcw}.
\end{proof}
\begin{remark}[Bridge Principle]\label{rm:bridge}
More generally, by similar arguments one could prove the following {\em bridge principle}: If $u^l\in H^1(\Om^l)\cap
 L^{\infty}(\Om^l)$ and $u^r\in H^1(\Om^r)\cap
 L^{\infty}(\Om^r)$ are isolated $L^1$-local minimizers of $F(\cdot, \Omega^l)$ and $F(\cdot, \Omega^r)$, respectively, then  there exists a family $(u_\e)$ such that $u_\e$ is an $L^1$-local minimizer of  $F(\cdot, \Om_\e)$ for $\e$ small enough and 
$$
\|u_\e(\e+\cdot, \cdot)-u^l\|_{L^1(\Om^l)}\to 0\,, \qquad \|u_\e(\cdot-\e, \cdot)-u^r\|_{L^1(\Om^r)}\to 0\,,
$$
as $\e\to 0^+$.  The local minimizers $u_\e$ can be constructed by the same constrained minimization procedure employed above; i.e., as  solutions to \eqref{tildeffe}, where $\widetilde F$ is defined as in \eqref{tildeffe0} and $\widetilde W$ satisfies {\em (a)} and {\em (b)} with $\alpha$ and $\beta$ replaced by $\|u^l\|_\infty$ and $\|u^r\|_\infty$, respectively, 
and $B_{d,\e}$ is as in \eqref{bde}, with $u_{0,\e}$ given by
$$
  u_{0,\e}(x,y):=
  \begin{cases}
  u^l & \text{if $(x,y)\in \Om_\e^l$,}\\
  \frac{(u^l+u^r)(0,0)}2 & \text{if $(x,y)\in N_\e$,}\\
  u^r & \text{if $(x,y)\in \Om_\e^r$.}\\
  \end{cases}
  $$
  Then, by similar arguments, one can show that  \eqref{lab**} still holds. We leave the details to the interested reader.
\end{remark}
  Next we show that given $\alpha$, $\beta\in V$, the corresponding admissible family of critical points as in Definition~\ref{def:gcw} is unique. More precisely, we have:
  \begin{theorem}[Uniqueness of nearly locally constant critical points]\label{th:uniqueness}
  Fix $\alpha$, $\beta\in V$ and $\overline M\geq\max\{|\alpha|, |\beta|\}$, and let  $\e_0>0$ and $\eta_0>0$ be the corresponding constants provided by  Theorem~\ref{th:locmin}. Then, there exixts $0<\e_1\leq\e_0$ depending only on $\alpha$, $\beta$ and $\overline M$ such that for all $0<\e\leq\e_1$ there is a unique critical point $u_\e$ of $F(\cdot, \Om_\e)$ with the property that $\|u_\e\|_{L^{\infty}(\Om_\e)}\leq \overline M$, $\|u_\e-\alpha\|_{L^1(\Om^l_\e)}\leq \frac{\eta_0}8$ and $\|u_\e-\beta\|_{L^1(\Om^r_\e)}\leq \frac{\eta_0}8$.
  \end{theorem}
  \begin{proof}
 Choose $\e_1\in (0, \e_0)$ be so small that $\frac{\eta_0}2+2\overline M|N_\e|<\eta_0$ for all $0<\e\leq\e_1$.  For $0<\e\leq \e_1$, let $u_\e$ and $v_\e$ be two critical points with all the required properties. Then, in particular, $\|u_\e-v_\e\|_{L^1(\Om_\e)}\leq  \|u_\e-\alpha\|_{L^1(\Om^l_\e)}+\|v_\e-\alpha\|_{L^1(\Om^l_\e)}+\|u_\e-\beta\|_{L^1(\Om^r_\e)}+\|v_\e-\beta\|_{L^1(\Om^r_\e)}+2\overline M|N_\e|\leq\frac{\eta_0}2+2\overline M|N_\e|< \eta_0$. Thus, by Theorem~\ref{th:locmin}, we have $F(v_\e, \Om_\e)> F(u_\e, \Om_\e)$ and  $F(u_\e, \Om_\e)> F(v_\e, \Om_\e)$, that is impossible.
  \end{proof}
As an immediate consequence of the previous theorem, we have:
 \begin{corollary}\label{cor:ab}
 If $(u_\e)_\e$ is a family of critical points as in Definition~\ref{def:gcw}, with $\alpha=\beta$, then for $\e$ small enough we have $u_\e\equiv\alpha$.
 \end{corollary}  
 If the potential $W$ satisfies (W3) of  Remark~\ref{rm:w3}, then the following holds.
   \begin{corollary}[Uniqueness under assumption (W3)]\label{cor:uniqueness}
   Assume that the potential  $W$ also satisfies  (W3) of Remark~\ref{rm:w3}. Then for any $\alpha$, $\beta\in V$ there exist $\e_1>0$ and $\eta_1>0$ such that for all $0<\e\leq\e_1$ there is a unique critical point $u_\e$ of $F(\cdot, \Om_\e)$ with the property that  $\|u_\e-\alpha\|_{L^1(\Om^l_\e)}\leq \eta_1$ and $\|u_\e-\beta\|_{L^1(\Om^r_\e)}\leq \eta_1$.
\end{corollary}
\begin{proof}
The statement is a straightforward consequence of Theorem~\ref{th:uniqueness}, after recalling that by Remark~\ref{rm:w3} any critical point has $L^\infty$-norm bounded by $\overline M$.
\end{proof}
 We conclude the section by showing that under convexity assumptions on the bulk regions $\Om^l$ and $\Om^r$ and some natural structural assumptions on the potential $W$,  a \emph{complete classification of stable critical points} can be given. To this aim, we recall the following notion of stability.
  \begin{definition}\label{def:stable}
A critical point $u_\e$ of $F(\cdot, \Om_\e)$ is called \em{stable}, if the second variation of 
$F(\cdot, \Om_\e)$ at $u_\e$ is non-negative definite; i.e., 
\beq\label{eq:stable}
\int_{\Om_\e}|\nabla\vphi|^2\, dxdy+\int_{\Om_\e}W''(u_\e)\vphi^2dxdy\geq 0 \qquad\text{for all } \vphi\in H^1(\Om_\e)\,.
\eeq
\end{definition}
We are now in a position to state the following result.
\begin{theorem}[Classification of stable critical points]\label{th:classification}
In addition to the standing hypotheses,  assume that $\Om^l$ and $\Om^r$ are smooth convex open sets,   that (W3) of Remark~\ref{rm:w3} holds, and that $W'(t)=0$ implies $W''(t)\neq0$. 
Then, there exists $\e_2>0$ such that for all $0<\e\leq\e_2$ the total number of non-constant stable critical points of 
$F(\cdot, \Om_\e)$ is given by $N(N-1)$, where $N:=\mathrm{card\,}V$. These stable critical points are nearly locally constant. More precisely, 
setting
$$
\eta_2:=\min_{\alpha_1\not=\alpha_2\in V}|\alpha_1-\alpha_2|\min\{|\Om^l|, |\Om^r|\}\,,
$$
 for each pair $(\alpha,\beta)\in V\times V$, with $\alpha\not=\beta$, and for $0<\e\leq\e_2$ there exists  a unique stable critical point $u_{\e}^{\alpha,\beta}$ of $F(\cdot, \Om_\e)$ such that $\|u_\e^{\alpha,\beta}-\alpha\|_{L^1(\Om^l_\e)}<\frac{\eta_2}2$ and $ \|u_\e^{\alpha,\beta}-\beta\|_{L^1(\Om^r_\e)}<\frac{\eta_2}2$. Viceversa, if $v$ is a non-constant stable critical point of 
 $F(\cdot, \Om_\e)$, with  $0<\e\leq\e_2$, then there exists a unique pair $(\alpha,\beta)\in V\times V$, with $\alpha\not=\beta$, such that $v=u_\e^{\alpha,\beta}$.   Moreover,
  $$
  \|u_\e^{\alpha,\beta}-\alpha\|_{L^1(\Om^l_\e)}\to 0\,, \qquad\text{and}\qquad \|u_\e^{\alpha,\beta}-\beta\|_{L^1(\Om^r_\e)}\to 0\,,
$$ 
as $\e\to 0$.
\end{theorem}
\begin{proof}
In view of Theorem~\ref{th:eu} and Corollary~\ref{cor:uniqueness}, the statement is an easy consequence of the following claim: {\it For all $\e>0$ sufficiently small let $u_\e$ be a non-constant stable critical point of $F(\cdot, \Om_\e)$. Then, then there exist $\alpha$, $\beta\in V$, with $\alpha\not=\beta$, such that, up to a subsequence, 
$$
  \|u_\e-\alpha\|_{L^1(\Om^l_\e)}\to 0\,, \qquad\text{and}\qquad \|u_\e-\beta\|_{L^1(\Om^r_\e)}\to 0\,.
$$  }
To this aim, we start by observing that, thanks to Remark~\ref{rm:w3}, the family $(u_\e)$ is uniformly bounded in $L^\infty$. Using \eqref{ELw} with $\vphi=u_\e$, we also have that the $H^1$-norms are uniformly bounded. Thus, we may find $u_0\in H^1(\Om^l\cup\Om^r)$ and a subsequence (not relabeled) such that 
\beq\label{uzero}
u_\e(\cdot+\e, \cdot)_{|\Om^r}\wto {u_0}_{|\Om^r}\quad\text{weakly in }H^1(\Om^r)\,, \qquad
u_\e(\cdot-\e, \cdot)_{|\Om^l}\wto {u_0}_{|\Om^l}\quad\text{weakly in }H^1(\Om^l)\,.
\eeq
Since the diameter of  $N_\e$ vanishes as $\e\to 0$, the 2-capacity of $N_\e$ vanishes as well. Therefore, it is possible to construct a family $(w_\e)$, with the following properties:
\begin{itemize}
\item[(a)] $w_\e\in H^1(\Om_\e)$ and $0\leq w_\e\leq 1$;
\item[(b)] $w_\e=0$ in $\Om_\e\setminus\Om_\e^r$;
\item[(c)]  $w_\e(x+\e,y)\to 1$ for a.e. $(x,y)\in \Om^r$; \vspace{2pt}
\item[(d)]  $\displaystyle\int_{\Om_\e}|\nabla w_\e|^2\, dxdy\to 0$ as $\e\to 0$.
\end{itemize}
Now we fix $\psi\in C^{\infty}(\overline{\Om}^r)$ and set
 $\vphi_\e:=w_\e \psi(\cdot-\e,\cdot)\in H^1(\Om_\e)$. By the criticality and the stability assumption,  recalling \eqref{eq:stable}, we have
\begin{align*} 
-& \int_{\Om_\e}\nabla u_\e\nabla\vphi_\e\, dxdy=\int_{\Om_\e}W'(u_\e)\vphi_\e\,dxdy\,,\\
& \int_{\Om_\e}|\nabla\vphi_\e|^2\, dxdy+\int_{\Om_\e}W''(u_\e)\vphi_\e^2dxdy\geq 0\,.
 \end{align*}
 Using \eqref{uzero}, the definition of $\vphi_\e$, and the properties of $w_\e$, one can check that in the limit as $\e\to 0$  the above expressions  become
 \begin{align*} 
-& \int_{\Om^r}\nabla u_0\nabla\psi\, dxdy=\int_{\Om^r}W'(u_0)\psi\,dxdy\,,\\
& \int_{\Om^r}|\nabla\psi|^2\, dxdy+\int_{\Om^r}W''(u_0)\psi^2dxdy\geq 0\,.
 \end{align*}
 Since $\psi$ is an arbitrary $C^\infty$ function on $\overline{\Om}^r$, by density we deduce that ${u_0}_{|\Om^r}$ is a stable critical point for $F(\cdot, \Om^r)$. In turn, by  \cite[Theorem 2]{CaHo}, the smoothness and the convexity of  $\Om^r$ imply that $u_0$ is a stable constant function; i.e., there  exists $\beta\in V$ such that ${u_0}_{|\Om^r}\equiv \beta$. The same argument shows that ${u_0}_{|\Om^l}\equiv \alpha$ for some $\alpha\in V$. Since
 all the $u_\e$ are non-constant, we must also have $\alpha\not=\beta$ thanks to Corollary~\ref{cor:ab}. This concludes the proof of the claim and the theorem follows.
\end{proof}
\section{Asymptotic behavior}

The goal of this section is to study the asymptotic behavior of admissible families $(u_\e)$ of nearly locally constant 
critical points as $\e\to 0$. As explained in the introduction, such a behavior is strongly influenced by the geometry of the neck $N_\e$ and, more specifically, by the asymptotic value of the ratio $\de\over\e$ between  width and length of $N_\e$. Before entering the details of the asymptotic analysis, we state and prove two technical lemmas that will be useful in the following. 
 
 \begin{lemma}
  Let $(u_\e)$ be a family of critical points as in Definition~\ref{def:gcw}. Then
    \beq\label{eub}
  F(u_\e, \Om_\e) - W(\alpha)|\Omega^l| - W(\beta)|\Omega^r| \leq \frac{C}{|\ln\delta|}\
  \eeq
 for some constant $C>0$ independent of $\e$. 

  \end{lemma}
  \begin{proof}
  Let  $\xi_{\e}$ be the test function constructed in the proof of Theorem~\ref{th:eu}. Since $\|u_\e -\xi_\e \|_{L^1(\Omega_\e)}\to 0$, by Theorem~\ref{th:locmin} we have that   $F (u_\e, \Omega_\e) \leq F(\xi_\e, \Omega_\e)$, provided that  $\e$ is sufficiently small. An explicit calculations shows that 
  \beq\label{recovery}
  \lim_{\e\to 0}|\ln\de|( F(\xi_\e, \Om_\e) - W(\alpha)|\Omega^l| - W(\beta)|\Omega^r|)=\frac{(\beta-\alpha)^2}{4}\frac{\pi}{1-\gamma}
  \eeq
  and the conclusion follows.
  \end{proof}
   \begin{lemma}[Barriers]\label{lm:barriers}
   For $0<\rho_0<\rho_1$  let $A^r(\rho_0, \rho_1):=\{(x,y):\rho_0<|(x,y)|<\rho_1,\, x>0\}$. Let 
   $u\in H^1(A^r(\rho_0, \rho_1))\cap L^{\infty}(A^r(\rho_0, \rho_1))$ satisfy
   $$
\begin{cases}
\Delta u=W'(u) & \text{in $A^r(\rho_0, \rho_1)$,}\\
\displaystyle \frac{\pa u}{\pa\nu}=0 & \text{on $\pa A^r(\rho_0, \rho_1)\cap\{x=0\}$,}\\
a_-\leq u\leq a_+ & \text{on $\pa B_{\rho_0}(0,0)\cap \{x>0\}$},\\
b_-\leq u\leq b_+ & \text{on $\pa B_{\rho_1}(0,0)\cap \{x>0\}$}
\end{cases} 
   $$
for some constants $a_\pm$ and  $b_\pm$.
 Let $d$ be any constant such that $d\geq\max_{|t|\leq\|u\|_{\infty}}|W'(t)|$. Then 
$$
u^-(x,y)\leq u(x,y)\leq 
  u^+(x,y)
$$
for all $(x,y)\in A^r(\rho_0, \rho_1)$, where
$$
u^\pm(x,y):= \frac{\mp d|(x,y)|^2}{4}+
\frac{(b_\pm-a_\pm)\pm\frac{d}4(\rho_1^2-\rho_0^2)}{\ln\frac{\rho_1}{\rho_0}}\ln\frac{|(x,y)|}{\rho_0}+
a_{\pm}\pm\frac d4\rho_0^2\,.
$$
  \end{lemma}
  \begin{proof}
  The conclusion follows by observing that  
  $$
\begin{cases}
\Delta u^\pm= \mp d& \text{in $A^r(\rho_0, \rho_1)$,}\\
\displaystyle \frac{\pa u^\pm}{\pa\nu}=0 & \text{on $\pa A^r(\rho_0, \rho_1)\cap\{x=0\}$,}\\
u^\pm=a_\pm  & \text{on $\pa B_{\rho_0}(0,0)\cap \{x>0\}$},\\
u^\pm=b_\pm & \text{on $\pa B_{\rho_1}(0,0)\cap \{x>0\}$\,}
\end{cases} 
   $$
   and by applying the comparison principle (see, for instance, \cite[Proposition 6.1]{MorSla}).
  \end{proof}

We are now in position to perform the asymptotic analysis in the various regimes.
 \subsection{The normal neck regime}
 In this subsection we consider the normal neck regime; i.e., we assume that 
 \beq\label{eq:normal}
 \lim_{\e\to0}\frac{\de}{\e}=\ell\in (0,\infty)\,.
 \eeq
We  denote by $\Om_\infty$ the ``limit'' of the rescaled sets 
$\frac1\e\Om_\e$. More precisely, $\Omega_\infty$ consists of the union of two half planes (the limits of the rescaled bulk domains) and the rescaled neck  
$$
\Om_\infty:=
 \Om_\infty^l\cup N_\infty\cup  \Om_\infty^r\,,
$$
where $\Om_\infty^l:=\{(x,y):\, x<-1\}$,  $\Om_\infty^r:=\{(x,y):\,x>1\}$,  and 
$N_\infty:=\{(x,y):\, |x|\leq1,\, -\ell f_2(x)<y<\ell f_1(x)\}$ (see Figure~\ref{fig:limitnormal} below).
 \begin{figure}[htbp]
  \begin{center}
  \includegraphics[scale=0.5]{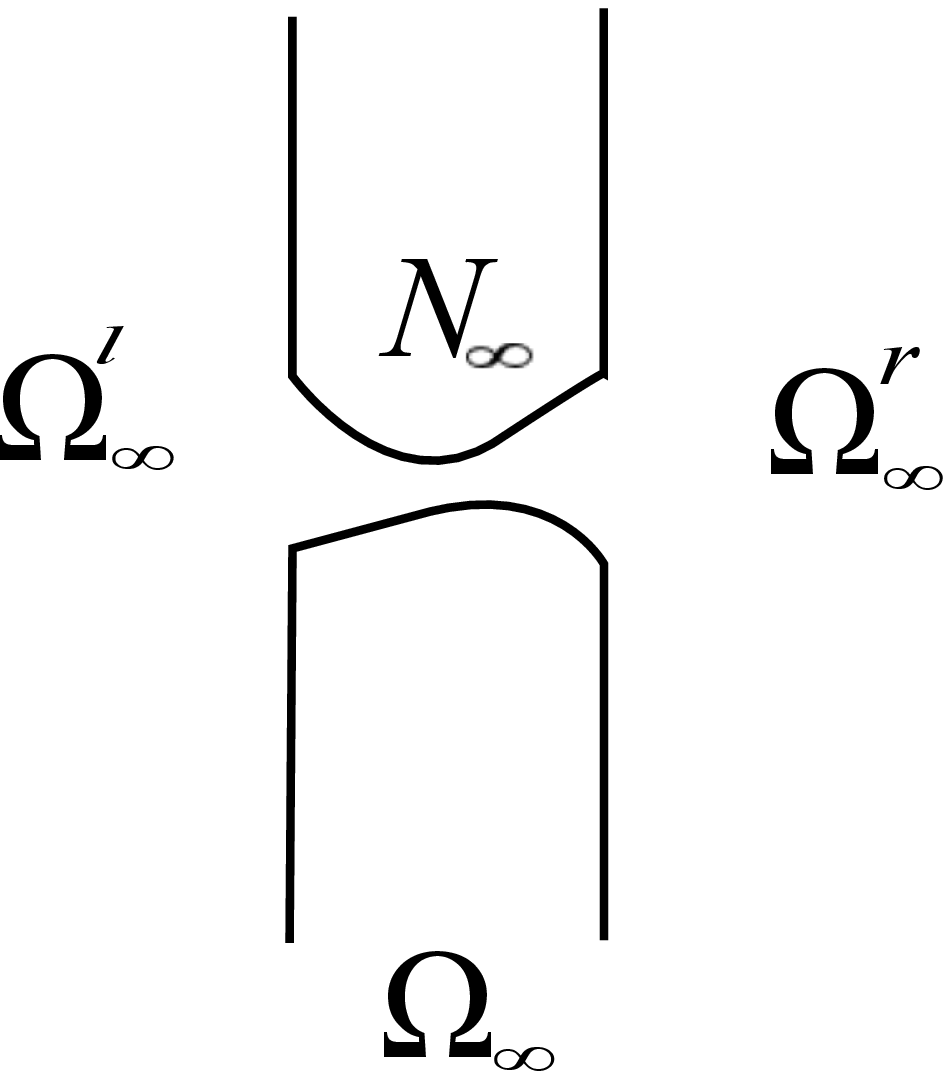}
  \caption{The limiting set $\Om_\infty$.}
  \label{fig:limitnormal}
  \end{center}
  \end{figure}
  We are now in a position to state the main result of this subsection.
 \begin{theorem}[Asymptotic behavior in the normal neck regime]\label{th:normal}
 Assume \eqref{eq:normal} and let $(u_\e)$ be a family of critical points as in Definition~\ref{def:gcw}.  Set 
 \beq\label{val}
 v_\e (x,y):= |\ln \e| (u_\e (\e x, \e y) - u_\e (0,0)).
 \eeq
 Then,  for every $p\geq 1$ we have $v_{\e}\to v$  in $W^{2,p}_{loc} (\Om_\infty)$ \footnote{Note that the local convergence of $v_\e$ to $v$ is well defined. Indeed since $\R^2\setminus \frac{1}{\e}\Om_\e\to \R^2\setminus \Om_\infty$ in the Kuratowski sense, it follows that for every $\Omega'\subset\subset\Om_\infty$ we have $\Omega'\subset\subset \frac{1}{\e}\Om_\e$ for $\e$ sufficiently small.}as $\e\to 0^+$, where $v$ is the unique solution to the 
 following problem:
 \beq\label{norm-limpob}
 \begin{cases}
 \Delta v=0 & \text{in $\Om_\infty$,}\vspace{4pt}\\
 \partial_\nu v=0 & \text{on $\partial  \Om_\infty$,}\vspace{4pt}\\
 \displaystyle\frac{v(x,y)}{\ln|(x,y)|}\to {\beta - \alpha \over 2} & \text{as $ |(x,y)|\to +\infty$ with $x>1$,}\vspace{4pt}\\
 \displaystyle\frac{v(x,y)}{\ln|(x,y)|}\to {\alpha - \beta \over 2} & \text{as $|(x,y)|\to +\infty$ with $x<1$,}\vspace{4pt}\\
v(0,0)=0\,.
 \end{cases}
 \eeq
Moreover, $u_\e (0, 0) \to {\alpha + \beta \over 2}$ and $\nabla v_\e \chi_{\frac1\e \Om_\e} \to \nabla v \chi_{\Om_\infty} $ in $L^2_{loc} (\R^2;\R^2)$. Finally,
\beq\label{normalen}
\lim_{\e\to 0^+}|\ln\e|\left( F(u_\e, \Om_\e) - W(\beta) |\Om^r| - W(\alpha) |\Om^l| \right)={\pi \over 4} (\beta -\alpha)^2 \,.
\eeq
\end{theorem}
\begin{remark}\label{rm:normal}
{\em The theorem shows that the rescaled profiles of nearly locally constant critical points  (and their energy) display a \emph{universal asymptotic behavior}, which depends only on the wells $\alpha$ and $\beta$, and on the limiting shape of the rescaled
necks. In particular, such a behavior is independent of  $\Om^l$, $\Om^r$, and the specific form
 of the double-well potential $W$. }
\end{remark}
\begin{proof}[Proof of Theorem~\ref{th:normal}]
To simplify the presentation and avoid inessential technicalities throughout the proof we assume $\ell=1$ and  $\delta = \e$. We also assume without loss of generality that $\alpha<\beta$.
Integrating by parts, we have
\beq \label{zerodir}
\int_{\Om_\e} |\nabla u_\e|^2 = \int_{\Om_\e} W'(u_\e) u_\e \to 0,
\eeq
where we have used the fact that $\chi_{\Om_\e}W'(u_\e) \to 0$ in $L^p$ for all $p\geq 1$, which easily follows from conditions (a)  and (b) of Definition~\ref{def:gcw}. In particular,
\beq\label{h1unif}
\sup_{0<\e\leq\bar\e}\|u_\e\|_{H^1(\Om_\e)}<+\infty\,.
\eeq
For any fixed $0<\rho_1<r_0$, let
\begin{align}\label{aer}
A^r_{\e}(\rho_1)&:=(\e,0)+\{(x,y)\in \Om^r:\rho_1<|(x,y)|\}\,,\\ 
A^l_{\e}(\rho_1)&:=-(\e,0)+\{(x,y)\in \Om^l:\rho_1<|(x,y)|\}\,.\nonumber
\end{align}
Recalling \eqref{h1unif} and the regularity assumptions on 
$\Om^l$ and $\Om^r$, by standard elliptic estimates we  have 
\beq\label{etae}
\eta^r_\e:=\|u_\e-\beta\|_{L^{\infty}(A^r_{\e}(\rho_1))}\to 0\,, \qquad 
\eta^l_\e:=\|u_\e-\alpha\|_{L^{\infty}(A^l_{\e}( \rho_1))}\to 0\,.
\eeq
We now split the remaining part of the proof into several steps.

\noindent{\bf Step 1.}({\it limit of $u_\e(0,0)$ and of the energy}) 
Set 
\beq\label{EMME}
M:=\max\{\|f_1\|_{\infty}, \|f_2\|_{\infty}\}+1\,
\eeq
 where $f_1$ and $f_2$ are the functions appearing in \eqref{neck}. 
Since the function $\hat u_\e(x,y):=u_\e (\e x, \e y)$  satisfies the Euler-Lagrange equation
$$
\begin{cases}
\Delta \hat u_\e  = \e^2 W' (\hat u_\e ), \\
\displaystyle {\partial \hat u_\e\over \partial n}   =0,
\end{cases}
$$
and   $\int_{B_{2M} (0,0)\cap\Om_{\infty}} | \nabla \hat u_\e |^2 \, dxdy\,  \to 0$ by  \eqref{zerodir}, again  standard regularity results imply the existence of a constant $m$ such that 
\beq\label{toemme0}
 \hat u_\e \to m\qquad\text{locally uniformly on $ B_{2M} (0,0) \cap \Om_\infty$} 
\eeq
and
\beq\label{toemme}
 \hat u_\e  \to m\qquad\text{ uniformly on $\partial B_{M} (1,0) \cap \Om_\infty$.} 
\eeq
Here we have also used the fact that $(\Om_\e/\e)\cap B_{2M}(0,0)=\Om_\infty\cap B_{2M}(0,0)$ for $\e$ small enough (see Assumption (O3) in Section~\ref{sec:formulation}). 
We claim that 
\beq\label{emme}
m=\frac{\alpha+\beta}{2}\,.
\eeq
To this aim, recall that for any given $\gamma\in (0,1)$, it is possible to construct  a sequence of functions $\xi_\e$ such that 
$\|\xi_\e-u_{\e}\|_{L^1(\Om_\e)}\to 0$ and
$$
\lim_{\e\to 0^+}|\ln\e|\left( F(\xi_\e, \Om_\e) - W(\beta) |\Om^r| - W(\alpha) |\Om^l| \right)=\frac{(\beta -\alpha)^2}4\frac{\pi} {1-\gamma} \,,
$$
see \eqref{recovery}. By \eqref{eq:locmin} and the arbitrariness of $\gamma$ we deduce that
\beq\label{limsup}
\limsup_{\e\to 0^+}|\ln\e|\left( F(u_\e, \Om_\e) - W(\beta) |\Om^r| - W(\alpha) |\Om^l| \right)\leq{\pi \over {4}} (\beta -\alpha)^2\,.
\eeq
Recall now that due to \eqref{toemme} for any given $\eta>0$ and $\e$ sufficiently small  we have
\beq\label{toemme2}
m-\eta\leq u_\e\leq m+\eta \qquad\text{on $\{(x,y):\, |(x-\e,y)|=M\e\,, x>\e\}$}.
\eeq
Moreover,
by \eqref{etae},
$$
\beta-\eta_\e^r\leq u_\e\leq \beta+\eta_\e^r \qquad\text{on $\{(x,y):\, |(x-\e,y)|=\rho_1\,, x>\e\}$.}
$$
Assume now that $m<\beta$  so that for $\eta$ and $\e$ sufficiently small we also have 
$m+\eta<\beta-\eta_\e^r$. Then,  we may estimate
 \begin{align}
 & \int_{\{M\e<|(x-\e,y)|<\rho_1,\, x>\e\}}|\nabla u_\e|^2 dxdy\nonumber \\
  &\geq \min\left\{\int_{\{M\e<|(x-\e,y)|<\rho_1,\, x>\e\}}\!\!\!|\nabla u|^2 dxdy:\, u\leq m+\eta\text{ on } 
 \pa B_{M\e}(\e, 0)\cap\{x>\e\}\,, \right. \nonumber\\
  &\left. \vphantom{\int}\qquad\qquad\qquad\qquad\quad u\geq \beta-\eta^r_\e\text{ on }  \pa B_{\rho_1}(\e, 0)\cap\{x>\e\}    \right\}\nonumber\\
  &= \min\left\{\int_{\{M\e<|(x-\e,y)|<\rho_1,\, x>\e\}}\!\!\!|\nabla u|^2 dxdy:\, u= m+\eta\text{ on } 
 \pa B_{M\e}(\e, 0)\cap\{x>\e\}\,, \right. \nonumber\\
  &\left. \vphantom{\int}\qquad\qquad\qquad\qquad\quad u= \beta-\eta^r_\e\text{ on }
   \pa B_{\rho_1}(\e, 0)\cap\{x>\e\}   \right\}\,,
   \label{limenbulk}
 \end{align}
where the last equality easily follows by a standard truncation argument, recalling that 
$m+\eta<\beta-\eta_\e^r$. The unique minimizer of the last minimization problem is given by
$$
 \tilde u_\e(x,y)=m+\eta +\frac{\beta-\eta^r_\e-m-\eta}{\ln\frac{\rho_1}{M\e}}\ln\frac{|(x-\e,y)|}{M\e}\,.
$$
The explicit computation of its Dirichlet energy, \eqref{limenbulk}, and the arbitrariness of $\eta$ yield
$$
\liminf_{\e\to 0}\frac{|\ln \e|}2\int_{\{M\e<|(x-\e,y)|<\rho_1,\, x>\e\}}|\nabla u_\e|^2 dxdy\geq 
\frac{\pi(\beta-m)^2}{2}\,.
$$
The same inequality is trivial when $m=\beta$ and can be proven similarly when $m>\beta$, using the fact that for $\eta$ and $\e$
sufficiently small $m-\eta>\beta+\eta_\e^r$. By an analogous argument we also have
 $$
\liminf_{\e\to 0}\frac{|\ln \e|}2\int_{\{M\e<|(x+\e,y)|<\rho_1,\, x<-\e\}}|\nabla u_\e|^2 dxdy\geq 
\frac{\pi(\alpha-m)^2}{2}\,.
$$
Collecting the two inequalities, we get
\beq\label{liminfdir}
\liminf_{\e\to 0}\frac{|\ln \e|}2\int_{\Om_\e} |\nabla u_\e|^2 dxdy\geq
\frac{\pi(\beta-m)^2}{2}+\frac{\pi(\alpha-m)^2}{2}\,.
\eeq
We now claim that 
\beq\label{nonnegative}
\liminf_{\e\to 0^+}|\ln\e|\left( \int_{\Om_\e}W(u_\e)\, dxdy - W(\beta) |\Om^r| - W(\alpha) |\Om^l| \right)\geq 0\,.
\eeq
To this aim, choose $\tau>0$  so small that that
\beq\label{tau}
W(t)\geq W(\alpha) \quad\text{for all $t\in (\alpha-\tau, \alpha+\tau)$,}\qquad
W(t)\geq W(\beta) \quad\text{for all $t\in (\beta-\tau, \beta+\tau)$.}
\eeq
This is possible thanks to condition (b) of Definition~\ref{def:gcw}. 

Recalling \eqref{toemme2} (with $\eta=1$) and \eqref{etae}, we can apply Lemma~\ref{lm:barriers} with
$\rho_0:=M\e$, $a_-:=m-1$, $b_-:=\beta-\eta_\e^r$, and 
$$
d:=\max_{|t|\leq \sup_{\e}\|u_\e\|_{\infty}}|W'(t)|
$$
to deduce that 
\beq\label{ue-}
u_\e(x,y)\geq u^-_\e(x,y):=u^-(x-\e,y)
\qquad\text{for $(x,y)\in \{M\e\leq |(x-\e,y)|\leq \rho_1,\, x>\e\}$}\,,
\eeq
where
$$
u^{-}(x,y):=\frac{d|(x,y)|^2}{4}
+\frac{\left(\beta-\eta_\e^r-m+1-\frac{d}{4}(\rho_1^2-M^2\e^2)\right)}{\ln\frac{\rho_1}{M\e}}\ln\frac{|(x,y)|}{ M\e}
+m-1-\frac d4M^2\e^2\,.
$$
Fix $\gamma\in (0,1)$ and note that 
$$
u_\e^-\geq
\begin{cases}
\displaystyle(\beta-\eta_\e^r-m+1)\frac{\ln\frac{\e^\gamma}{ M\e}}{\ln\frac{\rho_1}{M\e}}-\frac{d}2\rho_1^2+m-1
 & \text{if $\beta-\eta_\e^r-m+1\geq 0$}\vspace{5pt}\\
\displaystyle\vphantom{\int}\beta-\eta_\e^r-\frac{d}2\rho_1^2 & 
\text{otherwise}
\end{cases}
$$
on the set $\{\e^\gamma\leq |(x-\e,y)|\leq\rho_1\}$, 
provided that $\e$ is sufficiently small. By taking $\gamma$, $\rho_1$, and $\e$  small enough and recalling \eqref{etae} and \eqref{ue-}, we may conclude that
$$
u_\e\geq u_\e^- \geq \beta-\tau \qquad\text{on $\{\e^\gamma\leq |(x-\e,y)|\leq\rho_1\}$.}
$$
Using now the upper bound $u_\e^+:=u^+(\cdot-\e,\cdot)$ provided by Lemma~\ref{lm:barriers} with $a_+:=m+1$, 
$b_+:=\beta+\eta_\e^r$ and $\rho_0=M\e$, and $d$ as before (and taking $\gamma$ and $\rho_1$ smaller, if needed), we  can prove similarly that 
$$
u_\e\leq u_\e^+ \leq \beta+\tau \qquad\text{on $\{\e^\gamma\leq |(x-\e,y)|\leq\rho_1\}$.}
$$
Taking into account also \eqref{etae}, we therefore conclude that for $\e$ small enough
\beq\label{tau2}
\beta-\tau\leq u_\e\leq \beta+\tau  \qquad\text{on $A_\e^r(\e^\gamma)$,}
\eeq
where $A_\e^r(\e^\gamma)$ is the set defined in \eqref{aer} (with $\rho_1$ replaced by $\e^\gamma$).
Clearly,  the same argument  shows also that (upon possible modification of $\gamma$ and $\rho_1$, if necessary)
\beq\label{tau3}
\alpha-\tau\leq u_\e\leq \alpha+\tau  \qquad\text{on $A_\e^l(\e^\gamma)$}
\eeq
for all $\e$ sufficiently small. Combining \eqref{tau}, \eqref{tau2}, and \eqref{tau3}, we obtain
\begin{align*}
\int_{\Om_\e}W(u_\e)\, dxdy & =\int_{A_\e^l(\e^\gamma)}W(u_\e)\, dxdy+\int_{A_\e^r(\e^\gamma)}W(u_\e)\, dxdy+
\int_{\Om_\e\setminus\left(A_\e^l(\e^\gamma)\cup A_\e^r(\e^\gamma)\right)}W(u_\e)\, dxdy\\
& \geq W(\alpha)|\Om^l| +W(\beta)|\Om^r|-C \e^{2\gamma}\,,
\end{align*}
for some constant $C>0$ independent of $\e$. Note that we have also used the fact that the measure of 
$\Om_\e\setminus\left(A_\e^l(\e^\gamma)\cup A_\e^r(\e^\gamma)\right)$ is of order $\e^{2\gamma}$ together with the uniform $L^{\infty}$
bound on $W(u_\e)$.
From the above inequality we easily infer \eqref{nonnegative}.

Combining \eqref{limsup}, \eqref{liminfdir}, and \eqref{nonnegative} we obtain
\begin{align*}
\liminf_{\e\to 0^+}|\ln\e|&\left( F(u_\e, \Om_\e) -  W(\beta) |\Om^r| - W(\alpha) |\Om^l| \right)  \geq 
\frac{\pi(\beta-m)^2}{2}+\frac{\pi(\alpha-m)^2}{2}\\
& \geq {\pi \over {4}} (\beta -\alpha)^2 \geq
\limsup_{\e\to 0^+}|\ln\e|\left( F(u_\e, \Om_\e) - W(\beta) |\Om^r| - W(\alpha) |\Om^l| \right)\,.
\end{align*}
Hence, in particular,
$$
\frac{\pi(\beta-m)^2}{2}+\frac{\pi(\alpha-m)^2}{2}= {\pi \over {4}} (\beta -\alpha)^2 \,,
$$
which implies \eqref{emme} and \eqref{normalen}.

\noindent{\bf Step 2.} ({\it localization estimate for the energy}) 
 Let
\beq \label{ce0}
c_\e := \int_{\Om_\e \cap B_{2M\e} (0,0)} |\nabla u_\e|^2 \, dxdy.
\eeq
We claim that there exist positive constants $C_1$ and $C_2$ independent of $\e$ such that 
\beq \label{ce}
{C_1 \over |\ln \e|^2}\leq c_\e \leq { C_2 \over |\ln \e|^2}. 
\eeq
We argue by contradiction assuming that, up to a subsequence, either 
\beq\label{cenot}
c_\e |\ln \e|^2 \to \infty \quad \text{ as $\e \to 0$}
\eeq 
or
\beq\label{cenot2}
c_\e |\ln \e|^2 \to 0\quad \text{ as $\e \to 0$.}
\eeq
We can define for $(x,y) \in \Om_\e / \e$
$$
w_\e (x,y) = { 1 \over \sqrt{c_\e}} \left( u_\e (\e x, \e y) - \bar u_\e \right),
$$
where $\bar u_\e: = \medintinrigo_{B_{M\e} (0,0) \cap \Omega_\e} u_\e\, dxdy\,$. 
Notice that for small $\e$ we have
\beq\label{grad=1}
\int_{\Om_\infty \cap B_{2M} (0,0)} |\nabla w_\e|^2 \, dxdy=1\,.
\eeq
Here we used also that fact that $\Om_\infty \cap B_{2M} (0,0)=\frac{1}{\e}\Om_\e \cap B_{2M} (0,0)$ for $\e$ small enough. By compactness and standard elliptic estimates, we may thus assume that, up to subsequences,
\beq\label{insideball}
w_\e \to w_0 \quad \text{in }W^{2,p}_{loc}({\Om}_\infty \cap B_{2M}(0,0))\quad\text{and}\quad  \sup_{\e}\|w_\e\|_{L^{\infty}({\Om}_\infty \cap B_{r}(0,0))}<+\infty \text{ for all $0<r<2M$.}
\eeq
Moreover the convergence is uniform away from the corner points of $\Om_\infty \cap B_{M}(1,0)$, so that in particular we have $w_\e \to w_0$ uniformly on  $\pa B_{M}(1,0)\cap \{x>1\}$.
Set $m_0:=\min_{\pa B_{M}(1,0)\cap \{x>1\}}w_0-1$ and $M_0:=\max_{\pa B_{M}(1,0)\cap \{x>1\}}w_0+1$. Thus,
for $\e$ small enough we have 
$$
m_0\leq w_\e\leq M_0 \qquad \text{on }\pa B_{M}(1,0)\cap \{x>1\}
$$
or, equivalently,
$$
m_0\sqrt{c_\e}+\bar u_\e\leq u_\e\leq M_0\sqrt{c_\e}+\bar u_\e  \qquad \text{on }\pa B_{M\e}(\e,0)\cap \{x>\e\}\,.
$$
We can now apply Lemma~\ref{lm:barriers} with
$\rho_0:=M\e$, $a_-:=m_0\sqrt{c_\e}+\bar u_\e$, $a_+:=M_0\sqrt{c_\e}+\bar u_\e$, $b_-:=\beta-\eta_\e^r$, $b_+:=\beta+\eta_\e^r$ and 
\beq\label{d}
d:=\max_{|t|\leq \sup_{\e}\|u_\e\|_{\infty}}|W'(t)|
\eeq
to deduce that 
\begin{multline}\label{ue-bis}
u^-_\e(x,y):=u^-(x-\e,y)\leq u_\e(x,y)\leq u^+_\e(x,y):=u^+(x-\e,y)
\\\text{for $(x,y)\in \{M\e\leq |(x-\e,y)|\leq \rho_1,\, x>\e\}$}\,,
\end{multline}
where
\begin{multline*}
u^{-}(x,y):=\frac{d|(x,y)|^2}{4}
+\frac{\left(\beta-\eta_\e^r-m_0\sqrt{c_\e}-\bar u_\e-\frac{d}{4}(\rho_1^2-M^2\e^2)\right)}{\ln\frac{\rho_1}{M\e}}
\ln\frac{|(x,y)|}{M\e}
\\+m_0\sqrt{c_\e}+\bar u_\e-\frac d4M^2\e^2
\end{multline*}
and 
\begin{multline*}
u^{+}(x,y):=-\frac{d|(x,y)|^2}{4}
+\frac{\left(\beta+\eta_\e^r-M_0\sqrt{c_\e}-\bar u_\e+\frac{d}{4}(\rho_1^2-M^2\e^2)\right)}{\ln\frac{\rho_1}{ M\e}}
\ln\frac{|(x,y)|}{M\e}
\\+M_0\sqrt{c_\e}+\bar u_\e+\frac d4M^2\e^2\,.
\end{multline*}
Assume now that \eqref{cenot} holds. Then, it is straightforward to check that 
$$
\frac{u^-_\e(\e\cdot,\e\cdot)-\bar u_\e}{\sqrt{c_\e}}\to m_0\,, \quad \frac{u^+_\e(\e\cdot,\e\cdot)-\bar u_\e}{\sqrt{c_\e}}\to M_0\quad\text{locally uniformly in 
$\{x>1\}\setminus B_M(1,0)$.}
$$
Since
$$
\frac{u^-_\e(\e\cdot,\e\cdot)-\bar u_\e}{\sqrt{c_\e}}\leq w_\e\leq \frac{u^+_\e(\e\cdot,\e\cdot)-\bar u_\e}{\sqrt{c_\e}}
$$
and recalling \eqref{insideball}, we deduce that $w_\e$ are locally uniformly bounded in 
$\Om_\infty\cap\{x>0\}$. A completely analogous argument shows that the same locally uniform bounds hold in $\Om_\infty\cap\{x<0\}$. Therefore, by standard arguments (see for instance \cite[Proposition 6.2]{MorSla}) we can conclude that, up to a subsequence,  $w_\e\to w_0$ in 
$W^{2,p}_{loc}(\Om_\infty)$ for all $p>2$, where $w_0$ is a bounded harmonic function in $\Om_\infty$ satisfying homogeneous Neumann boundary conditions on $\pa \Om_\infty$. Using  the Riemann mapping theorem we can find a conformal mapping $\Psi$ from
the
infinite strip $\mathcal{R}:=(-1, 1)\times\R$ onto  $\Om_\infty$. Thus, $w_0\circ\Psi$ is bounded and 
harmonic in $\mathcal{R}$ and satisfies a homogeneous Neumann condition on $\partial\mathcal{R}$.
By reflecting  $w_0\circ\Psi$ infinitely many times, we obtain a bounded  entire harmonic function, which then must be constant by Liouville theorem. 
Since we also have $\nabla w_\e \chi_{\frac{\Om_\e}{\e}}\to \nabla w_0\chi_{\Om_\infty}$ in 
$L^2_{loc}(\R^2; \R^2)$ (see again \cite[Proposition 6.2]{MorSla}), it follows, in particular, 
$$
\int_{\Om_\infty \cap B_{2M} (0,0)} |\nabla w_\e|^2 \, dxdy=\int_{\frac{\Om_\e}{\e}\cap B_{2M} (0,0)} |\nabla w_\e|^2 \, dxdy\to 
\int_{\Om_\infty \cap B_{2M} (0,0)} |\nabla w_0|^2 \, dxdy=0\,,
$$
a contradiction to \eqref{grad=1}. 

We now assume that \eqref{cenot2} holds. Using also the fact that 
\beq\label{media}
\bar u_\e\to \frac{\alpha+\beta}2\,,
\eeq
which follows from \eqref{toemme0} and \eqref{emme}, one can check in this case that 
$$
w_\e\geq \frac{u^-_\e(\e x,\e y)-\bar u_\e}{\sqrt{c_\e}}\to +\infty
$$
for all $(x,y)\in \{x>1\}\setminus B_M(1,0)$. This, in turn, gives a contradiction to \eqref{insideball} and concludes the proof of \eqref{ce}.

\noindent{\bf Step 3.} ({\it conclusion}) 
Set  now
$$
\widetilde w_\e(x,y):=|\ln \e|(u_\e(\e x,\e y)-\bar u_\e)\,.
$$
Using \eqref{ce}, it follows that
$$
\int_{\Om_\infty \cap B_{2M} (0,0)} |\nabla \widetilde w_\e|^2 \, dxdy\leq C
$$
for some constant $C$ independent of $\e$. Thus, 
arguing exactly as before,  we may deduce the existence of $\widetilde w_0$ such that,  up to subsequences,
\beq\label{insideball2}
\widetilde w_\e \to \widetilde w_0 \quad \text{in }W^{2,p}_{loc}({\Om}_\infty \cap B_{2M}(0,0)) \text{ and }  \sup_{\e}\|\widetilde w_\e\|_{L^{\infty}({\Om}_\infty \cap B_{r}(0,0))}<+\infty \text{ for $0<r<2M$.}
\eeq
Moreover, again exactly as before, we may also show that 
$$
|\ln \e|(\widetilde u^-_\e(\e\cdot,\e\cdot)-\bar u_\e)\leq \widetilde w_\e\leq |\ln \e|(\widetilde u^+_\e(\e\cdot,\e\cdot)-\bar u_\e)\,,
$$
where $\widetilde u^-_\e$ and $\widetilde u^+_\e$ are defined as $u^-_\e$ and $u^+_\e$, respectively, with $c_\e$, $m_0$, and $M_0$ replaced by $\frac{1}{|\ln \e|^2}$,  $\widetilde m_0:=\min_{\pa B_{M}(1,0)\cap \{x>1\}}\widetilde w_0-1$, and $\widetilde M_0:=\max_{\pa B_{M}(1,0)\cap \{x>1\}}\widetilde w_0+1$, respectively. By a straightforward computation, taking into account \eqref{media}, we have that 
\beq\label{finalbounds}
\begin{array}{c}
|\ln \e|(\widetilde u^-_\e(\e x,\e y)-\bar u_\e)\to \displaystyle\widetilde m_0+\frac{\beta-\alpha}{2}\ln \frac{|(x-1, y)|}M\,, \vspace{5pt}\\
 |\ln \e|(\widetilde u^+_\e(\e x,\e y)-\bar u_\e)\to \displaystyle\widetilde M_0+\frac{\beta-\alpha}{2}\ln \frac{|(x-1, y)|}M
\end{array}
\eeq
for all $(x,y)\in \{x>1\}\setminus B_M(1,0)$. The convergence is in fact uniform on the bounded subsets of $\{x>1\}\setminus B_M(1,0)$. Recalling that $\widetilde w_\e$ satisfies
$$
\begin{cases}
\Delta \widetilde w_\e=|\ln \e|\e^2W'(u_\e) & \text{in $\frac{\Om_\e}{\e}$,} \\
\displaystyle\frac{\pa \widetilde w_\e}{\pa\nu}=0  & \text{on $\frac{\pa \Om_\e}{\e}$,}
\end{cases}
$$
 using \eqref{insideball2},  \eqref{finalbounds}, and the corresponding bounds in
 $\{x<-1\}\setminus B_M(-1,0)$,
 by \cite[Proposition 6.2]{MorSla} we can deduce that, up to subsequences, 
 \beq\label{almostdone0}
 \widetilde w_\e\to \widetilde w_0 \quad\text{in $W^{2,p}_{loc}(\Om_\infty)$,}
 \eeq 
 with $\widetilde w_0$ solving
\beq\label{almostdone}
 \begin{cases}
 \Delta \widetilde w_0=0 & \text{in $\Om_\infty$,}\vspace{4pt}\\
\displaystyle\frac{\pa \widetilde w_0}{\pa\nu} =0 & \text{on $\partial  \Om_\infty$,}\vspace{4pt}\\
 \displaystyle\frac{\widetilde w_0(x,y)}{\ln|(x,y)|}\to {\beta - \alpha \over 2} & \text{as $ |(x,y)|\to +\infty$ with $x>1$,}\vspace{4pt}\\
 \displaystyle\frac{\widetilde w_0(x,y)}{\ln|(x,y)|}\to {\alpha - \beta \over 2} & \text{as $|(x,y)|\to +\infty$ with $x<-1$.}
 \end{cases}
 \eeq
Next we claim that
\beq\label{claimissimo}
|\ln\e||\bar u_\e-u_\e(0,0)|\leq C
\eeq
for some constant $C$ independent of $\e$. To this aim, fix $0<\rho_0<M$ so small that
$B_{\rho_0}(0,0)\subset\subset \Om_\infty$ and define 
$a_\e:=\medintinrigo_{B_{\rho_0}(0,0)} u_\e(\e x, \e y)\, dxdy$. Notice that
\begin{align}
|\bar u_\e-a_\e|& \leq \medint_{B_{M}(0,0) \cap \Omega_\infty}|u_\e(\e x, \e y)-a_\e|\, dxdy \nonumber\\
& \leq C\|u_\e(\e\cdot, \e\cdot)-a_\e\|_{L^2(B_M(0,0) \cap \Omega_\infty)}\leq \frac{C}{|\ln\e|}\,,\label{claimissimo1}
\end{align}
where the least inequality follows from the Poincar\'e-Wirtinger inequality, \eqref{ce0}, and \eqref{ce}.  Observe now that by the Sobolev Embedding Theorem and standard elliptic estimates, we have for any $p>2$
\begin{align*}
\|\nabla u_\e(\e\cdot, \e\cdot)\|_{C^0(B_{\rho_0}(0,0))}& \leq C\|u_\e(\e\cdot, \e \cdot)-\bar u_\e\|_{W^{2,p}(B_{\rho_0}(0,0))}\nonumber \\ 
&\le C(\|\e^2W'(u_\e)\|_{L^p(B_M(0,0)  \cap \Omega_\infty)}+\|u_\e(\e\cdot, \e \cdot)-\bar u_\e\|_{H^1(B_M(0,0)  \cap \Omega_\infty)})
\leq \frac{C}{|\ln\e|}\,, 
\end{align*}
where the last inequality follows again from \eqref{ce0} and \eqref{ce}. From the above inequality, it immediately follows that 
$$
|a_\e-u_\e(0,0)|\leq \frac{C}{|\ln\e|}\,,
$$
which together with \eqref{claimissimo1} yields
\eqref{claimissimo}. 

We are now ready to conclude. Indeed, by \eqref{almostdone0} and \eqref{claimissimo}, we have that, up to a further subsequence, the functions $v_\e$ defined in \eqref{val} converge to $v$ in $W^{2,p}_{loc}(\Om_\infty)$ for every $p\geq 1$, where $v$ solves \eqref{norm-limpob}. Since the solution to this problem is unique, as shown in Step 5 of the proof  of \cite[Theorem 3.1]{MorSla}, the convergence holds for the full sequence. Finally, the fact that
$u_\e(0,0)\to \frac{\alpha+\beta}2$ follows from \eqref{media} and \eqref{claimissimo}.

\end{proof}
\subsection{The thick neck regime}
In this subsection we state the result concerning the asymptotic behavior of  admissible families of critical point in the so-called thick neck regime. We omit the proof since it is similar (and in fact easier) to that of Theorem~\ref{th:normal}. We define
$$
y_i=\min\{ \, f_i(x),\,  x\in [-1,1] \,  \}  \quad\text{for }i=1,2\,.
$$
Using assumptions on $f_i(x)$ it is clear that $y_i >0$ for $i=1,2$.
\begin{figure}[htbp]
  \begin{center}
  \includegraphics[scale=0.5]{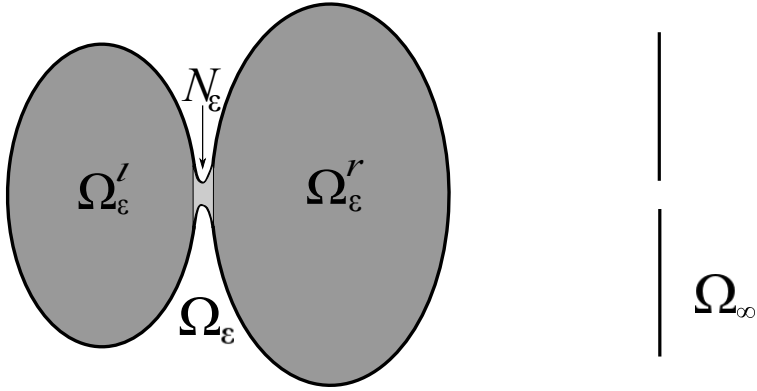}
  \caption{The limiting set $\Om^+_\infty$.}
  \label{fig:limitplus}
  \end{center}
  \end{figure}

 \begin{theorem}\label{th:thickneck}
 Assume that 
 $$
 \lim_{\e \to 0^+ }\frac{\de}\e=+\infty\,.
 $$
  Let $(u_\e)$ be a family of critical points as in Definition~\ref{def:gcw} and set
 $$
 v_\e(x,y):=
 |{\ln\de}| (u_\e(\de x,\de y)-u_\e(0,0))
 $$
 and $\Om_\infty:=\R^2\setminus\{(0,y):\,  y\geq y_1 \text{ or }y\leq -y_2\}$. 
 Then, for every $p\geq 1$  we have  $v_{\e}\to v$  in $W^{2,p}_{loc} (\Om_\infty)$, where $v$ is  the unique solution to  the  following problem:
 $$
 \begin{cases}
 \Delta v=0 & \text{in $\Om_\infty$,}\vspace{4pt}\\
 \partial_\nu v=0 & \text{on $\partial  \Om_\infty$,}\vspace{4pt}\\
 \displaystyle\frac{v(x,y)}{\ln|(x,y)|}\to \frac{\beta-\alpha}{2} & \text{as $ |(x,y)|\to +\infty$ with $x>0$,}\vspace{4pt}\\
 \displaystyle\frac{v(x,y)}{\ln|(x,y)|}\to \frac{\alpha-\beta}2 & \text{as $|(x,y)|\to +\infty$ with $x<0$,}\vspace{4pt}\\
v(0,0)=0\,.
 \end{cases}
$$
 Moreover, $u_\e(0, 0)\to \frac{\alpha+\beta}2$ and  $\nabla v_\e \chi_{\frac1\de \Om_\e} \to \nabla v \chi_{\Om_\infty} $ in $L^2_{loc} (\R^2;\R^2)$. Finally,
$$
 \lim_{\e\to 0^+} |\ln\de|(F(u_\e,\Om_\e)-W(\beta)|\Om^r|-W(\alpha)|\Om^l|)=\frac\pi4(\beta-\alpha)^2\,.
 $$
 \end{theorem}

\subsection{The thin neck regime}
We now consider the critical  thin  neck regime.
To simplify the presentation, we additionally assume that $f_1$ and $f_2$ are constant in a neighborhood of the points $-1$ and $1$. Precisely, there exists $\eta_0>0$ such that 
\beq\label{simple}
f'_i(x)=0 \quad \text{ for }x\in (1-\eta_0, 1)\cup (-1, -1+\eta_0)\,, \quad i=1,2\,.
\eeq
As it will be clear from the proof of the main result, the above assumption allows to avoid some technicalities in the construction of suitable lower and upper bounds and to present the main ideas in   a more transparent way. It could be removed by using the lower and upper bounds constructed in \cite{MorSla}, see Remark~\ref{rm:without} below.
 
 In order to state the next result, we set
\beq\label{mf1f2}
m_{_{f_1f_2}}:=\int_{-1}^1\frac1{f_1+f_2}\, dx\,.
\eeq

 \begin{figure}[htbp]
  \begin{center}
  \includegraphics[scale=0.4]{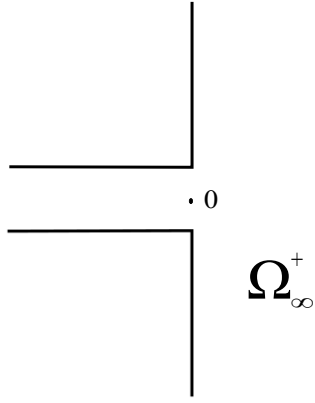}
  \caption{The limiting set $\Om^+_\infty$.}
  \label{fig:limitplus}
  \end{center}
  \end{figure}

\begin{theorem}[Critical thin neck]\label{th:thin-critical}
Assume that
\beq\label{eq:thincrit}
\lim_{\e\to 0^+}\frac{\de|\ln\de|}{\e}=\ell\in(0,+\infty)\,.
\eeq
Let $\{u_\e\}$ be the family of critical points as in Definition~\ref{def:gcw}. Then the following statements hold true.
\begin{enumerate}
\item[(i)] Let $\{v_\e\}$ be the family of rescaled profiles defined by
\beq\label{ve}
v_\e(x,y):=u_\e(\e x, \de y)\,.
\eeq
Then $v_\e\to v$  in $H^1(N)$, where $v(x,y):=\hat v(x)$ with $\hat v$ being the unique solution to 
the one-dimensional problem
\begin{multline}\label{1D-varbis}
\qquad\qquad\min\left\{\int_{-1}^1\frac{f_1+f_2}2(\theta^\prime)^2\, dx:\, \theta\in H^1(-1,1),\,\right. \\
\left.
 \theta(\pm1)=
\frac{\alpha+\beta}{2}\pm \frac{\pi\, m_{_{f_1f_2}}(\beta-\alpha)}{2(\pi\, m_{_{f_1f_2}}+ 2\ell)}\right\}\,,
\end{multline}
where $ m_{_{f_1f_2}}$ is the constant defined in \eqref{mf1f2}.
Moreover, 
\beq\label{energy2}
\lim_{\e\to 0^+}|\ln\de|F(u_\e, N_\e)=
 \frac{\ell\pi^2\,m_{_{f_1f_2}}(\beta-\alpha)^2}{ 2\left(\pi\,m_{_{f_1f_2}}+ 2\ell\right)^2}\,.
\eeq
\item[(ii)] Define
\beq\label{we+-}
w_\e^\pm(x,y):=|\ln\de|(u_\e(\de x\pm\e,\de y)-u_\e(\pm \e,0))\quad\text{for }(x,y)\in \widetilde \Om_\e^\pm:=\tfrac{\Om^\pm_\e+(\mp\e,0)}{\de}\,.
\eeq
Then,  
\beq\label{limce}
u_\e(\pm\e,0)\to \frac{\alpha+\beta}{2}\pm \frac{\pi\, m_{_{f_1f_2}}(\beta-\alpha)}{2(\pi\, m_{_{f_1f_2}}+ 2\ell)}\quad\text{as $\e\to 0^+$}
\eeq
 and the functions
$w_\e^\pm$ converge in $W^{2,p}_{loc}(\Om^\pm_\infty)$ for every $p\geq 1$ to the unique solution $w^\pm$ of the problem
\beq\label{we+asym}
\begin{cases}
\Delta w^\pm=0 & \text{in $\Om^\pm_\infty$,}\vspace{4pt}\\
\partial_{\nu} w^\pm=0 & \text{on $\partial\Om^\pm_\infty$, }\vspace{4pt}\\
 \displaystyle\frac{w^\pm(x,y)}{\ln|(x,y)|}\to \pm\frac{(\beta-\alpha)\ell}{\pi\,m_{_{f_1f_2}}+2\ell} & \text{as $ |(x,y)|\to \pm\infty$ with $\pm x>0$,}\vspace{4pt}\\
  \displaystyle\frac{w^\pm(x,y)}{x}\to \frac{1}{(f_1+f_2)(\pm1)}\frac{(\beta-\alpha)\ell\pi}{\pi\,m_{_{f_1f_2}}+2\ell}  & \text{uniformly in $y$ as $ x\to \mp\infty$,}\vspace{4pt}\\
  w^\pm(0,0)=0\,,
\end{cases}
\eeq
where (see Figure~\ref{fig:limitplus})
\beq\label{ominfty+}
\Om_\infty^\pm:=\left\{(x,y):\, \pm x\leq 0\,, -f_2(\pm1)<y< f_1(\pm1)\right\}\cup\{(x,y):\, \pm x>0\}\,.
\eeq
Moreover, $\nabla w_\e^\pm\chi_{\widetilde \Om_\e^\pm}\to \nabla w^\pm\chi_{\Om^\pm_\infty}$ in $L^2_{loc}(\R^2;\R^2)$. 

\item[(iii)] We have
\beq\label{energy3}
\lim_{\e\to 0^+}|\ln\de|\left( F(u_\e, \Om_\e\setminus N_\e)- W(\beta) |\Om^r| - W(\alpha) |\Om^l| \right)= \frac{(\beta-\alpha)^2\ell^2\pi}{\left(\pi\,m_{_{f_1f_2}}+2\ell\right)^2}\,.
\eeq
\end{enumerate}
\end{theorem}
For an interpretation of the boundary data $\theta(\pm 1)$ appearing in the one-dimensional minimum problem \eqref{1D-varbis}
in terms of a suitable limiting {\em renormalized energy} see Remark~\ref{rm:gamma} below. 
\begin{remark}\label{rm:thincrit}
The boundary conditions appearing in problem \eqref{1D-varbis}  show that only a part of the transition occurs inside the neck. The one-dimensional limiting profile described by \eqref{1D-varbis} is determined only by the shape of the neck itself. Note also that in \eqref{we+asym} the geometry is ``linearized'' and  the shape of the neck ``weakly'' affects the limiting bulk behavior only through the constant $m_{_{f_1f_2}}$ appearing in  the conditions at infinity.  We finally remark that the two conditions at infinity in \eqref{we+asym} are not independent,
 as  shown by Proposition~\ref{prop:we+asym} below.
\end{remark}

Before starting the proof of the theorem we recall the following proposition proved in \cite[Proposition 4.14]{MorSla}.
\begin{proposition}\label{prop:we+asym}
Let $\alpha$, $\beta>0$ and consider the set
$$
\Om^+_\infty:=\{(x,y):\, x\leq 0,\, |y|<\tfrac\alpha2\}\cup\{(x,y):\, x>0\}\,.
$$ 
Then, the problem
\beq\label{alphabeta}
\begin{cases}
\Delta w=0 & \text{in $\Om^+_\infty$,}\vspace{4pt}\\
\pa_\nu w=0 &  \text{on $\pa\Om^+_\infty$,}\vspace{4pt}\\
 \displaystyle\frac{w(x,y)}{\ln|(x,y)|}\to \beta & \text{as $ |(x,y)|\to +\infty$ with $x>0$,}\vspace{4pt}\\
\text{$w$ grows at most linearly   in $\Om^+_\infty\cap\{x< 0\}$,}\vspace{4pt}\\
w(0,0)=0
\end{cases}
\eeq
admits a unique solution. Moreover,
\beq\label{lim-infty}
\frac{w(x,y)}{x}\to \frac{\pi\beta}{\alpha}\qquad  \text{uniformly in $y$ as $ x\to -\infty$.}
\eeq
\end{proposition}
\begin{remark}
{We stress that  the previous statement  implies  that the logarithmic behavior of $w|_{\{x>0\}}$ at infinity, cobimbined with  the special one-dimensional geometry of the domain in $\{x<0\}$, uniquely  determine the linear asymptotic behavior of $w|_{\{x<0\}}$.}
\end{remark}

\begin{proof}[Proof of Theorem~\ref{th:thin-critical}]
We split the proof into several steps.

\noindent{\bf Step 1.} ({\it energy bounds in the neck}) 
First of all note that the same argument used to prove \eqref{nonnegative}, yields
\beq\label{nonnegativebis}
\begin{array}{c}
\displaystyle\liminf_{\e\to 0^+}|\ln\delta|\left( \int_{\Om^r_\e}W(u_\e)\, dxdy - W(\beta) |\Om^r|  \right)\geq 0 \\
\text{and}\\
\displaystyle   \liminf_{\e\to 0^+}|\ln\delta|\left( \int_{\Om^l_\e}W(u_\e)\, dxdy - W(\alpha) |\Om^l| \right)\geq 0\,.
\end{array}
\eeq
Considering the function $v_\e$ defined in
\eqref{ve}, and recalling \eqref{eub} and using \eqref{nonnegativebis}, it follows
\begin{align}
\int_{N_\e}|\nabla u_\e|^2\, dxdy&=
\int_{N_\e}\left[\frac{1}{\e^2}\Bigl|\partial_x v_\e\Bigl(\frac{x}{\e}, \frac{y}{\delta}\Bigr)\Bigr|^2+
\frac{1}{\delta^2}\Bigl|\partial_y v_\e\Bigl(\frac{x}{\e}, \frac{y}{\delta}\Bigr)\Bigr|^2\right]\, dxdy\nonumber\\
&= \int_{N}\left[\frac{\delta}{\e}|\partial_x v_\e(x, y)|^2+
\frac{\e}{\delta}|\partial_y v_\e(x, y)|^2\right]\, dxdy\leq \frac{C}{|\ln\de|}\,,\label{mezzocollo}
\end{align}
with $N$ defined in \eqref{unscaledN}.
Multiplying both sides of the last inequality by $\e/\delta$ and recalling \eqref{eq:thincrit}, we 
obtain
\beq\label{eq:neckbound}
\int_{N}\left[|\partial_x v_\e(x, y)|^2+
\frac{\e^2}{\delta^2}|\partial_y v_\e(x, y)|^2\right]\, dxdy \leq C
\eeq
for some constant $C>0$ independent of $\e$. Since $\e/\delta\to \infty$ as $\e\to 0$, by 
\eqref{eq:neckbound} we easily deduce that  $v_\e$ is bounded in $H^1(N)$ and, up to subsequences,  
\beq\label{subue}
v_\e\wto v\qquad\text{weakly in $H^1(N)$} 
\eeq
for some one-dimensional $v$ of the form
\beq\label{1Dv}
v(x,y)=\hat v(x)\qquad\text{with $\hat v\in H^1(-1,1)$.}
\eeq
 We will show that $\hat v$ is independent of the subsequence and solves \eqref{1D-varbis}. 

From \eqref{eq:thincrit}, \eqref{mezzocollo}, \eqref{subue}, and \eqref{1Dv} we have
\begin{align}
&\liminf_{\e\to 0}|\ln \delta| F(u_\e, N_\e)\geq\liminf_{\e\to 0}|\ln \delta| \frac12\int_{N_\e}|\nabla u_\e|^2\, dxdy\nonumber\\
&= \liminf_{\e\to 0} \frac12\int_{N}\left[\frac{\delta|\ln \delta|}{\e}|\partial_x v_\e(x, y)|^2+
\frac{\e|\ln\delta|}{\delta}|\partial_y v_\e(x, y)|^2\right]\, dxdy\nonumber\\
&\geq  \liminf_{\e\to 0} \frac\ell2\int_{N}|\nabla v_\e|^2\, dxdy\geq \frac\ell2\int_{N}|\nabla v|^2\, dxdy\nonumber\\
&=\frac\ell2\int_{-1}^1(f_1+f_2)(\hat v')^2\, dx\nonumber\\
&\geq\frac\ell2\min\left\{
\int_{-1}^1(f_1+f_2)(\theta^\prime)^2\, dx:\, \theta\in H^1(-1,1)\,, 
\theta(\pm 1)=\hat v(\pm1)\right\}\nonumber\\
&=\frac{\ell\left(\hat v(1)-\hat v(-1)\vphantom{\int}\right)^2}{2m_{_{f_1f_2}}}\,.
\label{limen-necktris}
\end{align}
The last equality follows from the explicit computation of the minimum problem.

 \noindent{\bf Step 2.} ({\it energy bounds in the bulk})   Let $\bar r>0$ satisfy $2\bar r<\min_{[-1,1]}f_i$, $i=1,2$. 
Since $\hat u_\e(x,y):=u_\e (\delta x+\e, \delta y)$  satisfies 
$$
\Delta \hat u_\e  = \delta^2 W' (\hat u_\e ), \qquad \text{in $B_{2\bar r}(0,0)$}
$$
and recalling that by  \eqref{zerodir}, $\int_{B_{2\bar r} (0,0)} | \nabla \hat u_\e |^2 \, dxdy\,  \to 0$, using standard regularity theory results we conclude that there exists a constant $m$ such that 
\beq\label{toemme0bis}
\hat u_\e \to m\qquad\text{uniformly on $B_{\bar r} (0,0)$.} 
\eeq
We claim that 
\beq\label{vhat}
 m=\hat v(1)\,.
 \eeq
 For this it is enough to observe that from \eqref{subue} and \eqref{1Dv} it easily follows that
 $ v_\e(\cdot, y)\wto \hat v$ weakly in $H^1(-1,1)$ for almost every $y\in (-2\bar r, 2\bar r)$.
 Thus, in particular, $v_\e(1, y)\to \hat v(1)$  for almost every $y\in (-\bar r, \bar r)$. Since $\hat u_\e(0, y)=v_\e(1,y)$, the claim follows from \eqref{toemme0bis}.
 We can now argue exactly as in the proof of \eqref{liminfdir} and use \eqref{nonnegativebis} to obtain that 
 \beq\label{lbcrit+}
 \liminf_{\e\to 0^+}|\ln\delta|\left( F(u_\e, \Om^r_\e) -  W(\beta) |\Om^r|  \right)  \geq 
\frac{\pi(\beta-\hat v(1))^2}{2}\,.
 \eeq 
 Analogously, one can show that 
  \beq\label{lbcrit-}
 \liminf_{\e\to 0^+}|\ln\delta|\left( F(u_\e, \Om^l_\e) -  W(\alpha) |\Om^l|  \right)  \geq 
\frac{\pi(\alpha-\hat v(-1))^2}{2}\,.
 \eeq 

\noindent{\bf Step 3.} ({\it asymptotic behavior in the neck and limit of the energy}) By \eqref{limen-necktris}, \eqref{lbcrit+}, and \eqref{lbcrit-} we have
\begin{multline}\label{liminfen-globbis}
\liminf_{\e\to 0}|\ln \delta|(F(u_\e, \Om_\e)-W(\alpha)|\Om^l|-W(\beta)|\Om^r|)\geq \\
\frac{\ell\left(\hat v(1)-\hat v(-1)\vphantom{\int}\right)^2}{2m_{_{f_1f_2}}}+
\frac{\pi(\alpha-\hat v(-1))^2}{2}+\frac{\pi(\beta-\hat v(1))^2}{2}
\geq \frac{(\beta-\alpha)^2\pi\ell}{2(m_{_{f_1f_2}}\pi+2\ell)}\,.
\end{multline}
Note that 
\begin{multline}\label{strictineqbis}
 \frac{\ell\left(\hat v(1)-\hat v(-1)\vphantom{\int}\right)^2}{2m_{_{f_1f_2}}}+
\frac{\pi(\alpha-\hat v(-1))^2}{2}+\frac{\pi(\beta-\hat v(1))^2}{2}=\frac{(\beta-\alpha)^2\pi\ell}{2(m_{_{f_1f_2}}\pi+2\ell)}\\ \iff \hat v(-1)=\frac{\alpha+\beta}{2}- \frac{\pi\, m_{_{f_1f_2}}(\beta-\alpha)}{2(\pi\, m_{_{f_1f_2}}+ 2\ell)}\, \quad \text{ and } \quad
\hat v(1)=\frac{\alpha+\beta}{2}+ \frac{\pi\, m_{_{f_1f_2}}(\beta-\alpha)}{2(\pi\, m_{_{f_1f_2}}+ 2\ell)}\,,
\end{multline}
as it easily follows by minimizing the function on the left-hand side with respect to $\hat v(-1)$ and $\hat v(1)$.
 On the other hand, for any fixed $\gamma\in (0,1)$  and for $M$ as in \eqref{EMME}, we may consider 
the test functions $z_\e$ defined as
\begin{multline*}
z_\e(x,y):=\\
\begin{cases}
\frac{1}\e\frac{\pi(\beta-\alpha)}{m_{_{f_1f_2}}\pi+2\ell}\int_{-\e}^x\frac{1}{(f_1+f_2)(\frac{s}\e)}\, ds+ \frac{\alpha+\beta}{2}- \frac{\pi\, m_{_{f_1f_2}}(\beta-\alpha)}{2(\pi\, m_{_{f_1f_2}}+ 2\ell)}& \text{in $N_\e$,}\vspace{6pt}\\
\frac{\alpha+\beta}{2}- \frac{\pi\, m_{_{f_1f_2}}(\beta-\alpha)}{2(\pi\, m_{_{f_1f_2}}+ 2\ell)} & \text{in $\{|(x+\e, y)|\leq M\delta,\, x<-\e\}$,}\vspace{6pt}\\
\frac{\alpha+\beta}{2}+ \frac{\pi\, m_{_{f_1f_2}}(\beta-\alpha)}{2(\pi\, m_{_{f_1f_2}}+ 2\ell)} & \text{in $\{|(x-\e, y)|\leq M\delta,\, x>\e\}$,}\vspace{6pt}\\
\frac{\ell(\beta-\alpha)}{m_{_{f_1f_2}}\pi+2\ell}\frac1{\left|\ln M\delta^{1-\gamma}\right|}\ln\frac{|(x+\e,y)|}{\delta^{\gamma}}+\alpha & \text{in $\{M\delta<|(x+\e,y)|<\delta^{\gamma}, \, x<-\e\}$,}\vspace{6pt}\\
\alpha & \text{otherwise in $\Om^{l}_\e$,}\vspace{6pt}\\
-\frac{\ell(\beta-\alpha)}{m_{_{f_1f_2}}\pi+2\ell}\frac1{\left|\ln M\delta^{1-\gamma}\right|}\ln\frac{|(x-\e,y)|}{\delta^{\gamma}}+\beta & \text{in $\{M\delta<|(x-\e,y)|<\delta^{\gamma}, \, x>\e\}$,}\vspace{6pt}\\
\beta & \text{otherwise in $\Om^{r}_\e$.}
\end{cases}
\end{multline*}
Taking into account the local minimality of $u_\e$, we have
\begin{align}
\limsup_{\e\to 0}|\ln \delta|&(F(u_\e, \Om_\e)-W(\alpha)|\Om^l|-W(\beta)|\Om^r|)\nonumber\\
&\leq \limsup_{\e\to 0}|\ln\delta|(F(z_\e, \Om_\e)-W(\alpha)|\Om^l|-W(\beta)|\Om^r|)
\nonumber\\
&\leq \lim_{\e\to 0}|\ln\de| \left(\frac12\int_{\Om_\e}|\nabla z_\e|^2\,dxdy+ \mathcal{L}^2
\left((N_\e\cup B_{\delta^{\gamma}}(\e,0)\cup B_{\delta^{\gamma}}(-\e,0))\right)\max_{[\alpha,\beta]}W\right)\nonumber\\
&=  \lim_{\e\to 0} |\ln\de|\frac12\int_{\Om_\e}|\nabla z_\e|^2\,dxdy= \frac{(\beta-\alpha)^2\pi\ell}{2(m_{_{f_1f_2}}\pi+2\ell)^2}\left(m_{_{f_1f_2}}\pi+\frac{2\ell}{1-\gamma}\right)\,,
\label{limsupenbis}
\end{align}
where the last equality follows by explicit computation of the Dirichlet energy of $z_\e$.

Combining \eqref{liminfen-globbis} and \eqref{limsupenbis}, since $\gamma$ can be chosen arbitrarily close to $0$,  we conclude
\begin{multline}\label{limenglobbis}
\lim_{\e\to 0}|\ln \delta|(F(u_\e, \Om_\e)-W(\alpha)|\Om^l|-W(\beta)|\Om^r|)=\\
\frac{\ell\left(\hat v(1)-\hat v(-1)\vphantom{\int}\right)^2}{2m_{_{f_1f_2}}}+
\frac{\pi(\alpha-\hat v(-1))^2}{2}+\frac{\pi(\beta-\hat v(1))^2}{2}
= \frac{(\beta-\alpha)^2\pi\ell}{2(m_{_{f_1f_2}}\pi+2\ell)}\,,
\end{multline}

which, in turn, yields 
\beq\label{hatv1bis}
\hat v(\pm 1)=\frac{\alpha+\beta}{2}\pm \frac{\pi\, m_{_{f_1f_2}}(\beta-\alpha)}{2(\pi\, m_{_{f_1f_2}}+ 2\ell)}
\eeq
thanks to \eqref{strictineqbis}.
Note that the last equality, together with \eqref{toemme0bis} and \eqref{vhat}, yields that 
$$
u_\e(\e,0)\to \frac{\alpha+\beta}{2}+ \frac{\pi\, m_{_{f_1f_2}}(\beta-\alpha)}{2(\pi\, m_{_{f_1f_2}}+ 2\ell)}\quad\text{as $\e\to 0^+$.}
$$
A completely similar argument holds for $u_\e(-\e, 0)$, thus proving  \eqref{limce}. 
Moreover, the limit in \eqref{limenglobbis} is independent of the selected subsequence and thus  the full sequence converges. Now, combining \eqref{limen-necktris}, \eqref{lbcrit+}, \eqref{lbcrit-},  \eqref{limenglobbis}, and
\eqref{hatv1bis}  one deduces that  all the inequalities in \eqref{limen-necktris}, \eqref{lbcrit+}, and\eqref{lbcrit-} are in fact equalities  and that, in turn, $\hat v$ solves \eqref{1D-varbis}. Hence,  $\hat v$ does not depend on the 
selected subsequence. In turn, the equalities in \eqref{limen-necktris}, \eqref{lbcrit+} and \eqref{lbcrit-} hold for the full sequence and prove \eqref{energy2} and \eqref{energy3}, respectively. 

The strong convergence in $H^1(N)$ of $\{v_\e\}$ to $v$ can now be proved easily using the convergence of the Dirichlet energy (see \cite[Theorem 4.3-page 664]{MorSla} for the details).

\noindent{\bf Step 4.} ({\it upper bound of the energy in small balls})  
Let $M$ be as in \eqref{EMME}. We claim that 
\beq\label{claimsb}
\int_{B_{2M\de}(\e,0)\cap\Om_\e}|\nabla u_\e|^2\, dxdy\leq\frac{C}{|\ln\de|^2}
\eeq
for some constant $C>0$ independent of $\e$.
To this aim, let
\beq \label{ce0bis}
c_\e := \int_{B_{2M\de}(\e,0)\cap\Om_\e} |\nabla u_\e|^2 \, dxdy
\eeq
and assume  by contradiction that, up to a subsequence, 
\beq\label{cenotbis}
c_\e |\ln \de|^2 \to \infty \quad \text{ as $\e \to 0$.}
\eeq 
Note that, thanks to (O3) and \eqref{simple},  
\beq\label{simple2}
 \text{for all $R>0$ }\quad B_{R}(0,0)\cap\Om^+_\infty=B_{R}(0,0)\cap\frac1\de(\Om_\e-(\e,0)) \text{ if $\e$ is sufficiently small.}
\eeq
Thus, we can define for $(x,y) \in B_{2M}(0,0)\cap\Om^+_\infty$ and for $\e$ sufficiently small
$$
w_\e (x,y) := { 1 \over \sqrt{c_\e}} \left( u_\e (\de x+\e, \de y) - \bar u_\e \right),
$$
where $\bar u_\e: = \medintinrigo_{B_{2M\de}(\e,0)\cap\Om_\e} u_\e\, dxdy\,$. 
Notice that  we have
\beq\label{tildebiszero}
\int_{ B_{2M}(0,0)\cap\Om^+_\infty} |\nabla w_\e|^2 \, dxdy=1\,.
\eeq
By compactness and standard elliptic estimates, we may thus assume that, up to subsequences,
\beq\label{insideballbis}
w_\e \to w_0 \text{ in }W^{2,p}_{loc}({\Om}^+_\infty \cap B_{2M}(0,0))\quad\text{and}\quad 
\sup_{\e}\|w_\e\|_{L^{\infty}({\Om}^+_\infty \cap B_{r}(0,0))}<+\infty \text{ for all $0<r<2M$.}
\eeq
Moreover, the convergence is uniform away from the corner points of $\Om^+_\infty \cap B_{M}(0,0)$, so that in particular we have $w_\e \to w_0$ uniformly on  $\Gamma:=(\pa B_{M}(0,0)\cup \{x=-1\})\cap\Om^+_\infty$.
Set $m_0:=\min_{\Gamma}w_0-1$ and $M_0:=\max_{\Gamma}w_0+1$. Thus,
for $\e$ small enough we have 
$$
m_0\leq w_\e\leq M_0 \qquad \text{on }\Gamma
$$
or, equivalently,
\beq\label{bd}
m_0\sqrt{c_\e}+\bar u_\e\leq u_\e\leq M_0\sqrt{c_\e}+\bar u_\e  \qquad \text{on }\de\Gamma+(\e,0)\,.
\eeq
Using \eqref{cenotbis} and arguing as for  \eqref{ue-bis}, we may now construct lower and upper bounds $u_\e^-$ and $u_\e^+$ such that 
$u_\e^-\leq u_\e\leq u_\e^+$ in $\{M\de\leq |(x-\e,y)|\leq \rho_1,\, x>\e\}$ for some fixed $\rho_1>0$, with $u_\e^-$ and $u_\e^+$ satisfying 
\beq\label{outsideball}
\frac{u^-_\e(\de\cdot+\e,\de\,\cdot)-\bar u_\e}{\sqrt{c_\e}}\to m_0\,, \quad \frac{u^+_\e(\de\cdot+\e,\de\,\cdot)-\bar u_\e}{\sqrt{c_\e}}\to M_0\quad\text{locally uniformly in 
$\{x>0\}\setminus B_M(0,0)$.}
\eeq
Notice that by \eqref{simple}, we have that $N_\e\cap\{(1-\frac{\eta_0}2)\e\leq x \leq-\de+\e\}$ has flat horizontal boundary.  Note also that 
\beq\label{bd2}
-C\leq u_\e\leq C\quad\text{on }\{x=(1-\tfrac{\eta_0}2)\e\}\cap N_\e\,,
\eeq
where $C$ is the constant appearing in Definition~\ref{def:gcw}. Let $d$ be as in \eqref{d} and note that 
$$
n_\e(x,y)^+:=C+
\frac{\frac{d}2(-\de+\frac{\eta_0}2\e)^2+\sqrt{c_\e}M_0+\bar u_\e-C}{-\de+\frac{\eta_0}2\e}(x- (1-\tfrac{\eta_0}2)\e)
-\frac{d}2(x- (1-\tfrac{\eta_0}2)\e)^2
$$
solves
$$
\begin{cases}
\Delta n_\e^+=-d & \text{on }N_\e\cap\{(1-\frac{\eta_0}2)\e\leq x \leq-\de+\e\}\,, \\
n_\e^+=C & \text{on }\{x=(1-\tfrac{\eta_0}2)\e\}\cap N_\e\,,\\
n_\e^+=M_0\sqrt{c_\e}+\bar u_\e & \text{on }\{x=-\de+\e\}\cap N_\e\,,\\
\partial_\nu n^+_\e=0 & \text{on }\pa N_\e\,,
\end{cases}
$$
while 
$$
n_\e(x,y)^-:=-C+
\frac{-\frac{d}2(-\de+\frac{\eta_0}2\e)^2+\sqrt{c_\e}m_0+\bar u_\e+C}{-\de+\frac{\eta_0}2\e}(x- (1-\tfrac{\eta_0}2)\e)
+\frac{d}2(x- (1-\tfrac{\eta_0}2)\e)^2
$$
satisfies
$$
\begin{cases}
\Delta n_\e^-=d & \text{on }N_\e\cap\{(1-\frac{\eta_0}2)\e\leq x \leq-\de+\e\}\,, \\
n_\e^-=-C & \text{on }\{x=(1-\tfrac{\eta_0}2)\e\}\cap N_\e\,,\\
n_\e^-=m_0\sqrt{c_\e}+\bar u_\e & \text{on }\{x=-\de+\e\}\cap N_\e\,,\\
\partial_\nu n^-_\e=0 & \text{on }\pa N_\e\,.
\end{cases}
$$
Thus, recalling \eqref{bd} and \eqref{bd2}, by the comparison principle we deduce that
\beq\label{n+-0}
n_\e^-\leq u_\e\leq n_\e^+ \qquad\text{on }N_\e\cap\{(1-\tfrac{\eta_0}2)\e\leq x \leq-\de+\e\}
\eeq
and in turn
\beq\label{n+-}
\frac{n_\e^-(\de \cdot+\e, \de\,\cdot)-\bar u_\e}{\sqrt{c_\e}}\leq w_\e\leq 
\frac{n_\e^+(\de \cdot+\e, \de\,\cdot)-\bar u_\e}{\sqrt{c_\e}} \qquad\text{on }
\Omega_\infty^+\cap\{-\tfrac{\eta_0}2\tfrac\e\de\leq x\leq -1\}\,.
\eeq
Using \eqref{cenotbis}, it is easy to check that 
\beq\label{n+-2}
\frac{n_\e^-(\de \cdot+\e, \de\,\cdot)-\bar u_\e}{\sqrt{c_\e}}\to m_0\,, \quad \frac{n_\e^+(\de \cdot+\e, \de\,\cdot)-\bar u_\e}{\sqrt{c_\e}}\to M_0\quad\text{locally uniformly in 
$\overline \Om_\infty^+\cap\{x\leq -1\}$.}
\eeq
Combining \eqref{insideballbis}, \eqref{outsideball}, \eqref{n+-}, and \eqref{n+-2}, we conclude that the functions $w_\e$ are locally uniformly bounded in $\bar\Om_\infty^+$. Therefore, by standard arguments (see \cite[Proposition 6.2]{MorSla}) we can infer that, up to subsequences, $w_\e \to w_0$ in $W^{2,p}_{loc}(\Om^+_\infty)$, $p>2$, where $w_0$ is a bounded harmonic function in $\Om^+_\infty$ satisfying homogeneous Neumann boundary conditions on $\pa \Om^+_\infty$. Using  the Riemann mapping theorem we can find a conformal mapping $\Psi$ from
the
infinite strip $\mathcal{R}:=(-1, 1)\times\R$ onto  $\Om^+_\infty$. Thus, $w_0\circ\Psi$ in bounded and 
harmonic in $\mathcal{R}$ and satisfies a homogeneous Neumann condition on $\partial\mathcal{R}$.
By reflecting  $w_0\circ\Psi$ infinitely many times, we obtain a bounded  entire harmonic function, which then must be constant by Liouville theorem. 
Since we also have $\nabla w_\e \chi_{\frac{\Om_\e-(\e,0)}{\de}}\to \nabla w_0\chi_{\Om^+_\infty}$ in 
$L^2_{loc}(\R^2; \R^2)$ (again by \cite[Proposition 6.2]{MorSla}), it follows, in particular, 
$$
\int_{\Om^+_\infty \cap B_{2M} (0,0)} |\nabla w_\e|^2 \, dxdy=\int_{\frac{\Om_\e}{\e}\cap B_{2M} (0,0)} |\nabla w_\e|^2 \, dxdy\to 
\int_{\Om^+_\infty \cap B_{2M} (0,0)} |\nabla w_0|^2 \, dxdy=0\,,
$$
a contradiction to \eqref{tildebiszero}. This concludes the proof of \eqref{claimsb}.

 \noindent{\bf Step 5.} ({\it asymptotic behavior in the bulk}) 
  Set  now
\beq\label{tildewe}
\widetilde w_\e(x,y):=|\ln \delta|(u_\e(\delta x+\e,\delta y)-\bar u_\e)\,,
\eeq
where, we recall $\bar u_\e= \medintinrigo_{B_{2M\delta} (\e,0)\cap\Om_\e} u_\e\, dxdy\,$ and $M$ is defined as in \eqref{EMME}. Observe that, thanks to \eqref{toemme0bis} and \eqref{limce}, we have
\beq\label{mediabis}
\bar u_\e\to \frac{\alpha+\beta}{2}+ \frac{\pi\, m_{_{f_1f_2}}(\beta-\alpha)}{2(\pi\, m_{_{f_1f_2}}+ 2\ell)}\quad\text{as $\e\to 0^+$.}
\eeq
Recalling \eqref{claimsb}, we also get
$$
\int_{B_{2M} (0,0)\cap\Om^+_{\infty}} |\nabla \widetilde w_\e|^2 \, dxdy\leq C
$$
for some constant $C$ independent of $\e$. 
Thus, 
arguing exactly as in the proof of \eqref{insideball2}--\eqref{almostdone}, we may construct suitable sub- and super-solutions and,  using \eqref{mediabis},   deduce the existence of $\widetilde w_0$ such that,    up to subsequences, 
 \beq\label{almostdone0bis}
 \widetilde w_\e\to \widetilde w_0 \quad\text{in $W^{2,p}_{loc}((B_{2M} (0,0)\cap\Om^+_{\infty})\cup \{x>0\})$,}
 \eeq 
 with $\widetilde w_0$ satisfying
\beq\label{almostdonebis}
 \begin{cases}
 \Delta \widetilde w_0=0 & \text{in $\{x>0\}$,}\vspace{4pt}\\
\displaystyle\frac{\pa \widetilde w_0}{\pa\nu} =0 & \text{on $\{x=0\}\setminus \{0\}\times (-f_2(1), f_1(1))$,}\vspace{4pt}\\
 \displaystyle\frac{\widetilde w_0(x,y)}{\ln|(x,y)|}\to \frac{(\beta-\alpha)\ell}{\pi\,m_{_{f_1f_2}}+2\ell} & \text{as $ |(x,y)|\to +\infty$ with $x>0$.}\vspace{4pt}
 \end{cases}
 \eeq

\noindent{\bf Step 6.} ({\it asymptotic behavior in the neck})
Note that by \eqref{n+-0}, we deduce
\beq\label{n+-final}
|\ln \de|(n_\e^-(\de \cdot+\e, \de\,\cdot)-\bar u_\e)\leq \widetilde w_\e\leq 
|\ln \de|(n_\e^+(\de \cdot+\e, \de\,\cdot)-\bar u_\e) \qquad\text{on }
\Omega_\infty^+\cap\{-\tfrac{\eta_0}2\tfrac\e\de\leq x\leq -1\}\,.
\eeq
Using \eqref{eq:thincrit} and \eqref{mediabis},  one can show  that 
\beq\label{n+-final2}
\begin{array}{c}
|\ln \de|(n_\e^-(\de x+\e, \de y)-\bar u_\e)\to \displaystyle m_0+\frac{2\ell}{\eta_0}\left(\frac{\alpha+\beta}{2}+ \frac{\pi\, m_{_{f_1f_2}}(\beta-\alpha)}{2(\pi\, m_{_{f_1f_2}}+ 2\ell)}+C\right)(x+1)\,, \vspace{5pt}\\
 |\ln \de|(n_\e^+(\de x+\e, \de y)-\bar u_\e)\to \displaystyle M_0+\frac{2\ell}{\eta_0}\left(\frac{\alpha+\beta}{2}+ \frac{\pi\, m_{_{f_1f_2}}(\beta-\alpha)}{2(\pi\, m_{_{f_1f_2}}+ 2\ell)}-C\right)(x+1)
\end{array}
\eeq
for all $(x,y)\in \{x<-1\}\cap \Om_\infty^+$. The convergence is in fact uniform on the bounded subsets of  $\{x<-1\}\cap \Om_\infty^+$.

Collecting \eqref{n+-final} and \eqref{n+-final2}, also from the previous step, we may infer that, up to subsequences,  the functions $\widetilde w_\e$ converge in $W^{2,p}_{loc}(\Om^\pm_\infty)$ for every $p\geq 1$ to the unique solution  $\tilde w_0$ of the problem
$$
\begin{cases}
\Delta \tilde w_0 =0 & \text{in $\Om^+_\infty$,}\vspace{4pt}\\
\partial_{\nu} \tilde w_0=0 & \text{on $\partial\Om^+_\infty$, }\vspace{4pt}\\
 \displaystyle\frac{\tilde w_0(x,y)}{\ln|(x,y)|}\to \pm\frac{(\beta-\alpha)\ell}{\pi\,m_{_{f_1f_2}}+2\ell} & \text{as $ |(x,y)|\to \pm\infty$ with $ x>0$,}\vspace{4pt}\\
  \text{$\tilde w_0$ grows at most linearly   in $\Om^+_\infty\cap\{x< 0\}$,}\vspace{4pt}\\
  \tilde w_0(0,0)=0\,,
\end{cases}
$$
Arguing as in the the final part of the proof of Theorem~\ref{th:normal}, the same convergence holds for the functions $w^+_\e$ defined in \eqref{we+-}. A completely analogous argument applies to the functions  $w^-_\e$. The conclusion  of the theorem follows from  Proposition~\ref{prop:we+asym}. 
\end{proof}
\begin{remark}\label{rm:without}
If one removes the extra assumption \eqref{simple}, the proof goes through without changes except for the construction of the lower and upper bounds $n_\e^-$ and $n^+_\e$ described in Step 4. In the general case, the construction of such barriers in the neck is more complicated and it is essentially performed in \cite[Lemmas 4.18 and 4.19]{MorSla}.
\end{remark}

\begin{remark}[Renormalized energy]\label{rm:gamma}
By considering the limit of the rescaled functionals
\beq\label{renF}
|\ln\delta|\left( F(u_\e, \Om_\e) - W(\beta) |\Om^r| - W(\alpha) |\Om^l| \right)
\eeq
we may introduce the following {\em renormalized limiting energy}, defined for all $(\theta_1, \theta_2)\in \R^2$ by
 \beq\label{RE}
 RE(\theta_1, \theta_2)= \frac{\ell\left(\theta_2-\theta_1\vphantom{\int}\right)^2}{2m_{_{f_1f_2}}}+
\frac{\pi(\alpha-\theta_1)^2}{2}+\frac{\pi(\beta-\theta_2)^2}{2}\,.
 \eeq
 Roughly speaking, the first term on the right-hand side represents the asymptotic optimal renormalized energy needed to make a transition from $\theta_1$ to
 $\theta_2$ inside the neck. The remaining two terms represent the optimal  bulk energy associated with transition from $\alpha$ to $\theta_1$ in the left bulk and from $\theta_2$ to $\beta$ in the right bulk, respectively.  In fact, by a slight modification of the arguments contained in the proof of Theorem~\ref{th:thin-critical}, one could show that the functionals \eqref{renF} $\Gamma$-converge to \eqref{RE} in the following sense:
 
\begin{itemize}
\item[(i)] (liminf inequality): Let $u_\e\in H^1(\Om_\e)$ and set 
$$
\theta_1(u_\e):=\medint_{B_{\delta}(-\e,0)\cap\Om_\e}u_\e\, dxdy\qquad\text{and}
\qquad
\theta_2(u_\e):=\medint_{B_{\delta}(\e,0)\cap\Om_\e}u_\e\, dxdy\,.
$$
 If $\|u_\e-\alpha\|_{L^{1}(\Om_\e^l)}\to 0$,   $\|u_\e-\beta\|_{L^{1}(\Om_\e^r)}\to 0$,  
 $\theta_1(u_\e)\to \theta_1$, and $\theta_2(u_\e)\to \theta_2$, then
$$
\liminf_{\e\to 0} |\ln\delta|\left( F(u_\e, \Om_\e) - W(\beta) |\Om^r| - W(\alpha) |\Om^l| \right)\geq 
 RE(\theta_1, \theta_2)\,.
$$ 
\item[(ii)] (limisup inequality): for every $(\theta_1, \theta_2)\in \R^2$, there exist a recovery sequence $u_\e\in H^1(\Om_\e)$ such that $\|u_\e-\alpha\|_{L^{1}(\Om_\e^l}\to 0$,   $\|u_\e-\beta\|_{L^{1}(\Om_\e^r}\to 0$,  
 $\theta_1(u_\e)\to \theta_1$, and $\theta_2(u_\e)\to \theta_2$, and 
 $$
\limsup_{\e\to 0} |\ln\delta|\left( F(u_\e, \Om_\e) - W(\beta) |\Om^r| - W(\alpha) |\Om^l| \right)\leq  RE(\theta_1, \theta_2)\,.
$$ 
\end{itemize}
Finally we notice that 
$$
\min_{\theta_1, \theta_2}RE(\theta_1, \theta_2)=\frac{(\beta-\alpha)^2\pi\ell}{2(m_{_{f_1f_2}}\pi+2\ell)}\,, 
$$
where the last quantity is exactly is the sum of the two limiting energies \eqref{energy2} and \eqref{energy3}. Moreover, the unique minimizers $\theta^{opt}_1$ and $\theta_{2}^{opt}$ coincide with the boundary data $\theta(-1)$ and $\theta(1)$, respectively, in the one-dimensional minimization problem \eqref{1D-varbis}.
\end{remark}

We conclude the section by stating the results for remaining thin neck regimes.
The asymptotic behavior can be formally  deduced from Theorem~\ref{th:thin-critical} by letting $\ell\to+\infty$ and $\ell\to 0$ respectively. We don't provide the proof here, since the result follows by similar arguments as in the proof of 
Theorem~\ref{th:thin-critical}, which in fact deals with the most difficult case.
We start by considering the subcritical case.
\begin{theorem}[Subcritical thin neck]\label{th:thin-subcritical}
Assume that
$$
\lim_{\e\to 0^+}\frac{\de|\ln\de|}{\e}=0\,.
$$
Let $\{u_\e\}$ be the family of critical point as in Definition~\ref{def:gcw} and  $\{v_\e\}$ be the family of rescaled profiles defined by
$$
v_\e(x,y):=u_\e(\e x, \de y)\,.
$$
Then $v_\e\to v$  in $H^1(N)$, where $v(x,y):=\hat v(x)$ with $\hat v$ being the unique solution to 
the one-dimensional problem
$$
\min\left\{\frac12\int_{-1}^1(f_1+f_2)(\theta^\prime)^2\, dx:\, \theta\in H^1(-1,1), \theta(-1)=\alpha\,, \theta(1)=\beta\right\}\,.
$$
Moreover, 
$$
\lim_{\e\to 0^+}\frac\e\de (F(u_\e, \Om_\e)-W(\beta)|\Om^r|-W(\alpha)|\Om^l|)=\lim_{\e\to 0^+}\frac\e\de F(u_\e, N_\e)=
 \frac{(\beta-\alpha)^2}{2m_{_{f_1f_2}}},.
$$
\end{theorem}
\begin{remark}\label{rm:subcritical-thin}
{Note the rescaled profiles $v_\e$ depend only on the shape of the neck.  The boundary conditions satisfied by $\hat v$ show that the whole transition from $\alpha$ to $\beta$  is asymptotically confined inside the neck.}
\end{remark}
We conclude with the supercritical case.
\begin{theorem}[Supercritical thin neck]\label{th:supercritical}
Assume that
$$
\lim_{\e\to 0^+}\frac{\de}{\e}=0\qquad\text{and}\qquad \lim_{\e\to 0^+}\frac{\de|\ln\de|}{\e}=+\infty\,.
$$
Let $\{u_\e\}$ be the family of critical points as in Definition~\ref{def:gcw}. Then the following statements hold true.
\begin{enumerate}
\item[(i)] Define
$$
w_\e^\pm(x,y):=|\ln\de|(u_\e(\de x\pm\e,\de y)-u_\e(\pm \e,0))\quad\text{for }(x,y)\in \widetilde \Om_\e^\pm:=\tfrac{\Om^\pm_\e+(\mp\e,0)}{\de}\,.
$$
Then,  
$$
u_\e(\pm\e,0)\to \frac{\alpha+\beta}{2}\quad\text{as $\e\to 0^+$}
$$
 and the functions
$w_\e^\pm$ converge in $W^{2,p}_{loc}(\Om^\pm_\infty)$ for every $p\geq 1$ to the unique solution $w^\pm$ of the problem
$$
\begin{cases}
\Delta w^\pm=0 & \text{in $\Om^\pm_\infty$,}\vspace{4pt}\\
\partial_{\nu} w^\pm=0 & \text{on $\partial\Om^\pm_\infty$, }\vspace{4pt}\\
 \displaystyle\frac{w^\pm(x,y)}{\ln|(x,y)|}\to \pm\frac{\beta-\alpha}{2} & \text{as $ |(x,y)|\to \pm\infty$ with $\pm x>0$,}\vspace{4pt}\\
  \displaystyle\frac{w^\pm(x,y)}{x}\to \frac{1}{(f_1+f_2)(\pm1)}\frac{(\beta-\alpha)\pi}{2}  & \text{uniformly in $y$ as $ x\to \mp\infty$,}\vspace{4pt}\\
  w^\pm(0,0)=0\,,
\end{cases}
$$
where 
$$
\Om_\infty^\pm:=\left\{(x,y):\, \pm x\leq 0\,, -f_2(\pm1)<y< f_1(\pm1)\right\}\cup\{(x,y):\, \pm x>0\}\,.
$$
Moreover, $\nabla w_\e^\pm\chi_{\widetilde \Om_\e^\pm}\to \nabla w^\pm\chi_{\Om^\pm_\infty}$ in $L^2_{loc}(\R^2;\R^2)$. 

\item[(ii)] We have
\begin{multline}\label{superenergy}
\lim_{\e\to 0^+}|\ln\de|\left( F(u_\e, \Om_\e)- W(\beta) |\Om^r| - W(\alpha) |\Om^l| \right)\\
=\lim_{\e\to 0^+}|\ln\de|\left( F(u_\e, \Om_\e\setminus N_\e)- W(\beta) |\Om^r| - W(\alpha) |\Om^l| \right)= \frac{(\beta-\alpha)^2\pi}{4}\,.
\end{multline}
\end{enumerate}
\end{theorem}
Note that in the supercritical case the whole transition occurs outside of the neck. This is also reflected in the limiting behavior of the energy \eqref{superenergy}.

\vskip 0.5cm
\noindent {\bf Acknowledgements}  VS would like to acknowledge support from the EPSRC grant EP/I028714/1

{
}

\end{document}